\documentclass{article}
\title{Constructing local models\\ for Lagrangian torus fibrations}
\author{Jonathan David Evans and Mirko Mauri}
\usepackage[utf8]{inputenc}
\usepackage[T1]{fontenc}
\usepackage{fixltx2e}
\usepackage{graphicx}
\usepackage{longtable}
\usepackage{float}
\usepackage{wrapfig}
\usepackage{rotating}
\usepackage{amsmath}
\usepackage{textcomp}
\usepackage{marvosym}
\usepackage{wasysym}
\usepackage{amssymb}
\usepackage{mathtools}
\usepackage{hyperref}
\usepackage{amsthm,amsmath,amsfonts,amscd}
\usepackage{tikz}
\usetikzlibrary{decorations.markings}
\usepackage{parskip}
\newcommand{\FF}{\mathfrak{f}}
\newcommand{\GG}{\mathfrak{g}}
\newcommand{\CC}{\mathbb{C}}
\newcommand{\QQ}{\mathbb{Q}}
\newcommand{\RR}{\mathbb{R}}

\newcommand{\ZZ}{\mathbb{Z}}
\newcommand{\Link}{\operatorname{Link}}
\newcommand{\Cone}{\operatorname{Cone}}
\newcommand{\Spec}{\operatorname{Spec}}

\newcommand{\nbhd}{\operatorname{nbhd}}
\newcommand{\cp}[1]{\mathbb{CP}^{#1}}

\newcommand{\OP}[1]{\mathrm{#1}}
\newcommand{\ttt}{\mathfrak{t}}

\begingroup
\makeatletter
\@for\theoremstyle:=definition,remark,plain\do{%
\expandafter\g@addto@macro\csname th@\theoremstyle\endcsname{%
\addtolength\thm@preskip\parskip
}%
}
\endgroup
\usepackage{graphicx}
\usepackage[capitalise]{cleveref}
\newtheorem{thm}{Theorem}[section]
\newtheorem{lem}[thm]{Lemma}

\newtheorem{Proposition}[thm]{Proposition}
\theoremstyle{remark}
\newtheorem{Remark}[thm]{Remark}
\theoremstyle{definition}
\newtheorem{Definition}[thm]{Definition}
\newtheorem{example}[thm]{Example}
\newtheorem{Assumption}[thm]{Assumption}
\crefname{thm}{Theorem}{Theorems}
\Crefname{thm}{Theorem}{Theorems}
\crefname{Lemma}{Lemma}{Lemmas}
\Crefname{Lemma}{Lemma}{Lemmas}
\crefname{lem}{Lemma}{Lemmas}
\Crefname{lem}{Lemma}{Lemmas}
\crefname{Corollary}{Corollary}{Corollaries}
\Crefname{Corollary}{Corollary}{Corollaries}
\crefname{Claim}{Claim}{Claims}
\Crefname{Claim}{Claim}{Claims}
\crefname{Proposition}{Proposition}{Propositions}
\Crefname{Proposition}{Proposition}{Propositions}
\crefname{Assumption}{Assumption}{Assumptions}
\Crefname{Assumption}{Assumption}{Assumptions}
\crefname{Remark}{Remark}{Remarks}
\Crefname{Remark}{Remark}{Remarks}
\crefname{Definition}{Definition}{Definitions}
\Crefname{Definition}{Definition}{Definitions}
\crefname{example}{Example}{Examples}
\Crefname{example}{Example}{Examples}
\crefname{Exercise}{Exercise}{Exercises}
\Crefname{Exercise}{Exercise}{Exercises}

\begin{document}
\maketitle
\begin{abstract}
We give a construction of Lagrangian torus fibrations with
controlled discriminant locus on certain affine varieties. In
particular, we apply our construction in the following ways:
\begin{itemize}
\item We find a Lagrangian torus fibration on the 3-fold negative vertex whose
discriminant locus has codimension 2; this provides a local model
for finding torus fibrations on compact Calabi-Yau 3-folds with
codimension 2 discriminant locus.
\item We find a Lagrangian torus fibration on a neighbourhood of the one-dimensional stratum of a simple
normal crossing divisor (satisfying certain conditions) such that
the base of the fibration is an open subset of the cone over the dual complex of the divisor. This can be used to construct an analogue of the
non-archimedean SYZ fibration constructed by Nicaise, Xu and Yu.

\end{itemize}
\end{abstract}
\section{Introduction}

The Strominger-Yau-Zaslow conjecture
\cite{SYZ,GrossWilson,KontsevichSoibelman} asserts that a Calabi-Yau
manifold admitting a degeneration to a large complex structure limit
also admits a Lagrangian torus fibration. Indeed, the fibres are
special Lagrangian and the Calabi-Yau metric should undergo
Gromov-Hausdorff collapse along the fibration as one approaches the
limit. If we take a minimal semistable reduction of the degeneration,
so that the singular fibre is a reduced simple normal crossing
variety, the base of the SYZ torus fibration should be the dual
complex\footnote{called {\em Clemens complex} in
\cite{KontsevichSoibelman}.} of the singular fibre {\cite[Section
3.3]{KontsevichSoibelman}}.

In this paper, we focus on the problem of finding Lagrangian torus
fibrations (not necessarily special) on affine manifolds \(W\)
arising as \(X\setminus Y\) where \(X\) is a complex projective
variety and \(Y\) is a simple normal crossing divisor. These should be
thought of as local models for Lagrangian torus fibrations on compact
Calabi-Yaus or other varieties (for example surfaces of general
type). The idea is very simple: construct a Lagrangian torus fibration
on the contact ``link at infinity'' of \(W\), extend it to the
interior of \(W\) using Liouville flow, and add a central fibre
corresponding to the Lagrangian skeleton of \(W\). If \(B\) is the
base of the Lagrangian torus fibration on the link then we are careful
to identify \(B\) with the dual complex of \(Y\). The base of the
Lagrangian torus fibration on \(W\) is then \(\Cone(B)\).

The singularities of the Lagrangian torus fibrations constructed in
this way are usually too severe to be of much use. However, in some
cases, this construction yields Lagrangian fibrations with mild
singularities which are difficult to construct any other way. We
demonstrate this with two examples:
\begin{itemize}
\item Mikhalkin \cite{Mikhalkin} introduced the {\em tropicalisation}
  map which, amongst other things, gives a topological torus fibration
  from a variety to its tropicalisation. In particular, one obtains a
  map from the 4-dimensional pair-of-pants to the cell complex which
  is the cone on the 1-skeleton of a tetrahedron. In Section
  \ref{sct:pop}, we give a version of this map with Lagrangian
  fibres. We anticipate that this will be useful for constructing
  Lagrangian torus fibrations on surfaces of general type.
\item We give a Lagrangian torus fibration on the affine variety
\begin{equation}\label{eq:negv}W=\{(x,y,u_1,u_2)\in\CC^2\times(\CC^*)^2\ :\
  xy=u_1+u_2+1\}\end{equation} such that the base is a 3-ball and the
discriminant locus (set of singular fibres) is a Y-graph. The
existence of such a Lagrangian torus fibration was conjectured by
Gross \cite{Gross}, who gave a topological torus fibration on this
space with the same discriminant locus. The singular fibre over the
vertex of the graph appears to be the same as the one Gross
conjectured (a space with Euler characteristic \(-1\) which we will
call the {\em Gross fibre}, see {\cite[Example 2.6.(4)]{Gross}}). 
\end{itemize}
This second example is of particular significance. In the early days
of mirror symmetry, there was an expectation that the discriminant
locus for a Lagrangian torus fibration on a Calabi-Yau 3-fold should
be a trivalent graph. The only fibres with nontrivial Euler
characteristic should live over the vertices of this graph, so there
should be two local models where the discriminant is a Y-graph: one
where the fibre over the vertex has Euler characteristic \(1\), and
one where it has Euler characteristic \(-1\) (as there are Calabi-Yau
3-folds with both positive and negative Euler characteristic). The
local model with Euler characteristic \(1\) is easy to construct (see
{\cite[Example 1.2]{GrossTop}}).

An explicit Lagrangian fibration on the variety given by Equation
\eqref{eq:negv} was found by Casta\~{n}o-Bernard and Matessi
\cite{CBM1,CBM2}, but the discriminant locus was a codimension 1
thickening of the Y-graph. Indeed, Joyce \cite{Joyce} had earlier
argued that the property of having codimension 1 discriminant locus
should be generic amongst singularities of special Lagrangian torus
fibrations. For this reason, attention in recent years has focused on
understanding the case where the discriminant locus has codimension
1. For example, the powerful method introduced by W.-D. Ruan
\cite{Ruan1,Ruan2,Ruan3} for producing Lagrangian torus fibrations on
compact Calabi-Yaus produces fibrations with codimension 1
discriminant locus unless special care is taken.

In light of our construction, it seems reasonable to expect
(continuous) Lagrangian torus fibrations on compact Calabi-Yau 3-folds
with codimension 2 discriminant locus, as originally expected.

An alternative approach to the SYZ conjecture is to look for
non-archimedean analogues of Lagrangian torus fibrations instead of
special Lagrangian fibrations. This was first suggested in {\cite[\S
  3.3]{KontsevichSoibelman}} and \cite{KontsevichSoibelman2006}, and
details of this approach are explained in \cite{NicaiseXuYu}. We give
an overview in Section \ref{sct:nonarch}, which explains how it
relates to the ideas in the current paper. According to
{\cite[Conjecture 1]{KontsevichSoibelman2006}}, the smooth part of the
Gromov-Hausdorff limit of maximally degenerating families of
Calabi-Yau varieties should carry an integral affine structure, of
which it should be possible to give a purely algebraic description;
see \cite[Conjecture 3]{KontsevichSoibelman2006} and \cite[Theorem
6.1]{NicaiseXuYu}. In Section \ref{sct:nonarch}, we show how the
integral affine structure constructed in \cite[Theorem
6.1]{NicaiseXuYu} arises essentially from a Lagrangian torus fibration
on an open subset of the total space of the degenerating family.

\subsection{Outline}

\begin{itemize}
\item In Sections \ref{sct:stratifications}--\ref{sct:ltf}, we give
  some basic definitions and setup; in particular, we define precisely
  what we mean by a Lagrangian torus fibration and a coisotropic
  fibration.
\item In Sections \ref{sct:ncd}--\ref{sct:sympnbhd}, we explain how to
  construct a Lagrangian torus fibration on an affine variety
  $X \setminus Y$, and we discuss some obstructions to lift Lagrangian
  torus fibrations from subvarieties of $Y$.
\item In Section \ref{sct:weinstein}, we apply the previous
  construction to find Lagrangian torus fibrations on the
  4-dimensional pair-of-pants and on the negative vertex (the variety
  defined by Equation \eqref{eq:negv}).
\item In Section \ref{sct:dualcomplex}, we discuss the dual complex of
  a simple normal crossing divisor and prove some results about the
  topology of the dual complex under the assumption that the
  complement of the simple normal crossing divisor is affine.
\item In Section \ref{sct:evaluationmap}, we construct a coisotropic
  fibration on the contact boundary of \(X\setminus Y\) such that the
  base of the fibration is the dual complex of the compactifying
  divisor \(Y\).

\item In Section \ref{sct:nonarch}, we explain a construction of
  Lagrangian torus fibrations with codimension 2 discriminant locus
  which is a symplectic analogue of the non-archimedean SYZ fibration
  constructed in \cite{NicaiseXuYu}.

\end{itemize}
\subsection{Acknowledgements}

We would like to acknowledge useful comments from the anonymous
referees, which have helped us improve the exposition in this paper.

JE is supported by EPSRC Standard Grant EP/P02095X/1 and would like to
acknowledge useful conversations with Ivan Smith, Paul Hacking, Peng
Zhou, Renato Vianna, Dmitry Tonkonog, Mark McLean and Mohammed
Tehrani.

MM is supported by the Max Planck Institute for Mathematics, and the
Engineering and Physical Sciences Research Council [EP/L015234/1]. The
EPSRC Centre for Doctoral Training in Geometry and Number Theory (The
London School of Geometry and Number Theory), University College
London. MM expresses his gratitude to his advisor Paolo Cascini for
his support and advice, and to Enrica Mazzon and Johannes Nicaise for
helpful discussions on Section \ref{sct:nonarch}.

\section{Definitions and setup}

\subsection{Stratifications}
\label{sct:stratifications}
\begin{Definition}\label{dfn:stratification}
Recall that a stratification of a topological space \(B\) is a
filtration
\[\emptyset=B_{-1}\subset B_0\subset\cdots\subset B_d\subset
B_{d+1}\subset\cdots\subset B,\]
where each \(B_d\) is a closed subset such that:
\begin{itemize}
\item the \(d\)-stratum \(S_d(B):=B_d\setminus B_{d-1}\) is a smooth
\(d\)-dimensional manifold for each \(d\) (possibly empty),
\item \(B=\bigcup_{d\geq 0}B_d\).
\end{itemize}
We say that \(B\) is finite-dimensional if the \(d\)-stratum is
empty for sufficiently large \(d\), and we say that \(B\) is
\(n\)-dimensional if \(B\) is finite-dimensional and \(n\) is the
maximal number for which \(S_n(B)\) is nonempty, in which case we
call \(S_n(B)\) the {\em top stratum}, \(S_{top}(B)\).

\end{Definition}
\begin{Definition}[Cones]\label{dfn:cone}
Given a stratified space \(B\), the {\em cone over \(B\)}, denoted
\(\Cone(B)\), is the stratified space
\[\Cone(B)=([-\infty,\infty)\times B)/(\{-\infty\}\times B),\] whose
strata are the open cones on the strata of \(B\) and the singleton
stratum comprising the cone point \([\{-\infty\}\times B]\).

\end{Definition}
\subsection{Lagrangian torus fibrations}
\label{sct:ltf}
We are interested in proper continuous maps \(F\colon X\to B\) with
connected fibres where \(B\) is a stratified space and, either:
\begin{itemize}
\item \((X,\omega)\) is a symplectic manifold of dimension \(2n\), or
\item \((X,\alpha)\) is a contact manifold of dimension \(2n-1\) with a
chosen contact 1-form \(\alpha\). In this case, we write
\(\omega=d\alpha\).

\end{itemize}
\begin{Definition}[Coisotropic fibration]\label{dfn:coisotropicfibration}
  We call \(F\colon X\to B\) a {\em coisotropic fibration} if it is a
  proper continuous map with connected fibres such that the fibres
  over the \(k\)-dimensional stratum are \(\omega\)-coisotropic
  submanifolds (possibly with boundary and corners) of codimension
  \(k\). We say a coisotropic fibration is {\em generically
    Lagrangian} if \(F\) is a smooth submersion over \(S_{top}(B)\)
  and the fibres over \(S_{top}(B)\) are Lagrangian. (Note that a
  Lagrangian submanifold of a contact manifold is one on which
  \(d\alpha\) vanishes.)

\end{Definition}
\begin{Remark}
If \(F\) is generically Lagrangian then \(\dim B=n\), respectively
\(\dim B=n-1\), if \(X\) is symplectic, respectively contact.

\end{Remark}
\begin{Definition}[Lagrangian torus fibration]\label{dfn:lagrangiantorusfibration}
We call \(F\colon X\to B\) a {\em Lagrangian torus fibration} if
\begin{itemize}
\item \(F\) is a smooth submersion over the top stratum, with Lagrangian
fibres.
\item the fibres over other strata are themselves stratified spaces with
isotropic strata.
\end{itemize}
We call the complement of the top stratum the {\em discriminant
locus} \(Discr(F)=B\setminus S_{top}(B)\).

\end{Definition}
\begin{Definition}[Integral affine structure]\label{dfn:Zaff}
Let \(M\) be a manifold of
dimension \(n\). The following definitions are equivalent.
\begin{enumerate}
\item An {\em integral affine structure} is a maximal atlas of \(M\)
  such that the transition functions belong to
  \(\operatorname{GL}(n, \ZZ)\rtimes \ZZ^n\).
\item An {\em integral affine structure} is a sheaf
  \(\operatorname{Aff}_M\) of continuous functions on \(M\) such that
  \((M, \operatorname{Aff}_M)\) is locally isomorphic to \(\RR^n\)
  endowed with the sheaf of degree one polynomials with integral
  coefficients.
\item An {\em integral affine structure} is the datum of a
  torsion-free flat connection on the tangent bundle $TM$, and of a
  lattice of flat sections of maximal rank in $TM$.
\end{enumerate}
See {\cite[\S 2.1]{KontsevichSoibelman2006}} for a proof of the
equivalence.
\end{Definition}
\begin{Remark}
The Arnold-Liouville theorem tells us that the Lagrangian fibres of
\(F\) are tori and, if \(X\) is symplectic, that the top stratum of
\(B\) inherits an integral affine structure.

\end{Remark}
\begin{Remark}
In the definition of coisotropic fibration, the fibres can be very
large. For example, the fibres over the 0-strata of \(B\) are
codimension zero submanifolds with boundary. One example of a
coisotropic fibration would be taking a Lagrangian torus fibration
\(X\to B\) and postcomposing with a map which collapses an open set
in \(B\) containing the discriminant locus. This motivates the
following definition:

\end{Remark}
\begin{Definition}\label{dfn:refinement}
We will say that a Lagrangian torus fibration \(F_1\colon X\to B_1\) is
a {\em refinement} of a coisotropic fibration \(F_2\colon X\to B_2\)
if there is a continuous surjection \(g\colon B_1\to B_2\) such that
\(F_2=g\circ F_1\).

\end{Definition}
We finish this section with some remarks on Lagrangian torus
fibrations on contact manifolds.

\begin{Remark}\label{rmk:reeb}
Recall that the Reeb flow for a contact manifold \((M,\alpha)\) is
always tangent to a \(d\alpha\)-Lagrangian submanifold of \(M\).

\end{Remark}
\begin{Remark}
Suppose that \(f\colon X\to B\) is a Lagrangian torus fibration on a
contact \((2n-1)\)-manifold \((X,\alpha)\). The action integrals
\(\int_{C_1}\alpha,\ldots,\int_{C_n}\alpha\) of the contact 1-form
around a basis for the homology of a general fibre define a map from
the universal cover of the top stratum \(S_{top}(B)\) to
\(\RR^n\). This map necessarily avoids the origin. Otherwise the
pullback of \(\alpha\) to the fibre is a nullhomologous closed
1-form, which therefore vanishes somewhere on the fibre. However it
must always evaluate positively on the Reeb vector, which is tangent
to the fibre by Remark \ref{rmk:reeb}.

\end{Remark}
\begin{Remark}
  In the context of the previous remark, the action integrals
  \(e^r\int_{C_1}\alpha\), \(\ldots\), \(e^r\int_{C_n}\alpha\) on the
  symplectisation give local action coordinates for the Lagrangian
  torus fibration \((r,f)\colon\RR\times X\to\RR\times B\) on the
  symplectisation \((\RR\times X,d(e^r\alpha))\).

\end{Remark}
\begin{Remark}
  The following rigidity result for Lagrangian torus fibrations on
  contact manifolds with particular conditions on the Reeb dynamics
  will apply to some examples later (see Remark \ref{rmk:rigidity}).
\end{Remark}

\begin{lem}\label{lma:rigidity}
Let \((M,\alpha)\) be a contact manifold with contact 1-form
\(\alpha\) and let \(F_1\colon M\to B_1\), \(F_2\colon M\to B_2\) be
Lagrangian torus fibrations. Suppose that for an open set \(U\subset
(B_1)_{top}\) of regular fibres there is a dense set \(V\subset U\)
such that, for any \(v\in V\), the fibre \(F_1^{-1}(v)\) contains a
dense Reeb orbit. Then \(F_2\) is constant along the fibres of
\(F_1\). In other words, the fibration \(F_2\) coincides with
\(F_1\) over \(F_1^{-1}(U)\).
\end{lem}
\begin{proof}
The fibre of \(F_2\) through a point \(x\in F_1^{-1}(v)\), \(v\in
V\), necessarily contains the closure of the Reeb orbit through
\(x\), so it contains the whole fibre \(F_1^{-1}(v)\) (recall that
we have a standing assumption that fibres are connected). Since this
is true for a dense set \(V\subset U\), and since \(F_1\) and
\(F_2\) are continuous, this means \(F_2|_{F_1^{-1}(u)}\) is
constant for all \(u\in U\), as required. \qedhere

\end{proof}
\subsection{Affine varieties}
\label{sct:ncd}

Let \(X\) be a smooth complex projective variety of complex dimension
\(n\). Let \(Y\subset X\) be a simple normal crossing divisor with
components \(Y_i\), \(i\in I\), for some indexing set \(I\). Suppose
that there is an ample divisor \(\sum_{i\in I} m_iY_i\), \(m_i\geq
1\), fully supported on \(Y\). Let \(L_i\) be the holomorphic line
bundle with \(c_1(L_i)=Y_i\), so that \(Y_i\) is the zero locus of a
transversely vanishing section \(s_i\in H^0(X,L_i)\). Let
\(|\cdot|_i\) be a Hermitian metric on \(L_i\) and let \(\nabla_i\) be
the corresponding Chern connection (uniquely determined by the
condition of being a metric connection for \(|\cdot|_i\) compatible
with the holomorphic structure of \(E\)) with curvature
\(F_{\nabla_i}\). By {\cite[p.148]{GH}}, we can choose the Hermitian
metrics in such a way that \(\omega:=-2\pi i\sum_{i\in
I}m_iF_{\nabla_i}\) is a K\"{a}hler form on \(X\) with K\"{a}hler
potential \[\psi:=-\sum_{i\in I}m_i\log|s_i|_i\colon W\to\RR\] defined
on \(W=X\setminus Y\), that is \(\omega=2\pi
i\partial\bar{\partial}\psi\). In other words, \(\psi\) is a
plurisubharmonic exhausting function on \(W\).

We now have a standard package of geometric objects associated to
\(\psi\):
\begin{itemize}
\item its gradient (with respect to the K\"{a}hler metric on \(W\)) is a
{\em Liouville vector field} \(Z:=\nabla\psi\); in other words
\(\mathcal{L}_Z\omega=\omega\). The 1-form
\(\lambda:=\iota_Z\omega\) is a primitive for \(\omega\) called the
\emph{Liouville form}. The flow \(\phi_t\) for time \(t\) along
\(Z\) is called the \emph{Liouville flow}: it dilates the symplectic
form in the sense that \(\phi_t^*\omega=e^t\omega\).
\item the critical locus of \(\psi\) is compact {\cite[Lemma
4.3]{Seidel}}.
\item if \(\psi\) is Morse or Morse-Bott, the union of all downward
manifolds of the Liouville flow is isotropic (called the
\emph{skeleton}, \(\Sigma\)); see {\cite[Proposition
11.9]{CieliebakEliashberg}}. Since the critical locus of \(\psi\) is
compact, this can be achieved by perturbing the Hermitian metrics on
the bundles \(L_i\) (and hence the symplectic form) over a compact
subset of \(W\).
\item if \(M\subset W\) is a \((2n-1)\)-dimensional submanifold transverse
to the Liouville flow then \(M\) inherits a contact form \(\alpha\)
pulled back from the Liouville form \(\lambda\). Moreover, if \(M\)
is disjoint from the skeleton \(\Sigma\) then the complement
\(W\setminus\Sigma\) is symplectomorphic to a subset of the
symplectisation \(SM\) of the form \(\{(r,x)\in\RR\times M\ :\
r<T(x)\}\), where \(T(x)=\sup\{t\ :\ \phi_t(x)\mbox{ is
defined}\}\). In particular, if \(Z\) is complete then \(W\setminus
\Sigma\) is symplectomorphic to the symplectisation of \(M\).

\end{itemize}
\begin{Definition}[Link at infinity]\label{dfn:link}
Let \(M\subset W\) be a \((2n-1)\)-dimensional submanifold
transverse to the Liouville flow and disjoint from the skeleton. The
contact manifold \((M, \lambda|_{M})\) is called \emph{link (at
infinity) of \(W\)}, denoted by \(\Link(W)\).

\end{Definition}
The contactomorphism type of \(\Link(W)\) is independent of the choice
of \(M\) by {\cite[Lemma 11.4]{CieliebakEliashberg}}.

\begin{lem}\label{lma:construction}
  Suppose there is a contact form \(\alpha\) on \(\Link(W)\) (not
  necessarily \(\lambda|_{\Link(W)}\)) for which there is a Lagrangian
  torus fibration \(f\colon \Link(W)\to B\). We get a Lagrangian torus
  fibration \(F\colon \overline{W}\to\Cone(B)\) where \(\overline{W}\)
  is the symplectic completion of \(W\).
\end{lem}
\begin{proof}
  The complement of the skeleton \(\overline{W}\setminus\Sigma\) is
  symplectomorphic to the symplectisation of the link:
  \(\left(\RR_t\times\Link(W),d\left(e^t\lambda|_{\Link(W)}\right)\right)\). Since
  \(\alpha=e^f\lambda|_{\Link(W)}\) for some function \(f\), this
  symplectisation is itself symplectomorphic to the symplectisation
  \(\left(\RR_\ttt\times\Link(W),d\left(e^\ttt\alpha\right)\right)\),
  via the symplectomorphism \((\ttt,x)\mapsto(t+f(x),x)\).
  
  The fibration \(F\colon\RR\times\Link(W)\to\RR\times B\),
  \(F(\ttt,x)=(\ttt,f(x))\) defines a Lagrangian torus fibration on
  the complement \(\bar{W}\setminus\Sigma\). This extends continuously
  to a map \(\overline{W}\to\Cone(B)\) by sending \(\Sigma\) to the
  cone point.
\end{proof}

\begin{Remark}
  The discriminant locus of \(F\colon W\to\Cone(B)\) is the cone on
  the discriminant locus of \(f\colon\Link(W)\to B\).
\end{Remark}

\begin{Remark}
  Although the existence of a Lagrangian torus fibration on a contact
  manifold depends on the contact form \(\alpha\) (the fibres must be
  \(d\alpha\)-Lagrangian), Lemma \ref{lma:construction} only depends
  on the contactomorphism (not strict contactomorphism) type of the
  link.
\end{Remark}

\begin{Remark}
  The Lagrangian skeleton of \(W\) is not easy to find in
  general. Ruddat, Sibilla, Treumann and Zaslow \cite{RSTZ} have given
  a conjectural description of the skeleton when \(W\) is a
  hypersurface in \((\CC^*)^m\times\CC^n\), and all the varieties we
  consider in this paper fall into this class. When \(W\) is
  hypersurface in \((\CC^*)^m\), P. Zhou \cite{Zhou} has confirmed
  that there is a Liouville structure whose skeleton is precisely the
  RSTZ skeleton.
\end{Remark}

\begin{Remark}
  In the situation of the SYZ conjecture, where \(W\) is a local model
  for a Calabi-Yau variety, \(\Cone(B)\) should be homeomorphic to a
  codimension zero submanifold of the \(n\)-sphere. One might wonder
  in that case if the map \(F\) is smooth {\em across} the different
  strata for some chosen smooth structure on \(\Cone(B)\). That is not
  a question we will consider in this paper, as it does not appear
  natural from our perspective: in our more general setting,
  \(\Cone(B)\) will not always be a topological manifold.
\end{Remark}

\subsection{Symplectic neighbourhoods}
\label{sct:sympnbhd}

Let \(Y\subset X\) be a smooth codimension 2 symplectic submanifold
with symplectic normal bundle \(\pi\colon\nu\to Y\). Pick a Hermitian
metric and a unitary connection \(\nabla\) on \(\nu\) and let
\(\mu\colon\nu\to\RR\) denote the function which generates the
rotation of fibres. Let \(\mathcal{H}\) denote the horizontal spaces
of \(\nabla\) and define a 2-form \(\Omega\) on \(\nu\) which:
\begin{itemize}
\item equals \(\pi^*\omega_Y+\langle\mu,F_\nabla\rangle\) on pairs of
  horizontal vectors (thinking of \(\mu\) as
  \(\mathfrak{u}(1)^*\)-valued and \(F_\nabla\) as
  \(\mathfrak{u}(1)\)-valued),
\item equals the fibrewise Hermitian volume form on pairs of vertical
  vectors,
\item makes \(\mathcal{H}\) orthogonal to the fibres.
\end{itemize}
This 2-form \(\Omega\) is closed and is nondegenerate on a
neighbourhood of the zero-section; see for instance \cite[Section 2.2]{MTZ}.
\begin{Definition}[{\cite[Definition 2.9]{MTZ}}]
  An \(\omega\)-regularisation of \(Y\) in \(X\) is a
  symplectomorphism
  \[\Psi\colon \left(\nbhd_\nu(Y),\Omega\right) \to
    \left(\nbhd_X(Y),\omega\right)\] from a neighbourhood of the
  zero-section in \(\nu\) to a neighbourhood of \(Y\) in \(X\) such
  that, along \(Y\), \(d\Psi\) induces the canonical isomorphism from
  the vertical distribution of \(\nu\) to the normal bundle of \(Y\).
\end{Definition}

The following lemma is immediate from the definition of \(\Omega\). In
the special case when \(F_\nabla=\omega\), it recovers Biran's circle
bundle construction.
\begin{lem}\label{lma:biran}
  Suppose that \(L\subset Y\) is a Lagrangian submanifold on whose
  tangent spaces the curvature \(F_\nabla\) vanishes. Then, for all
  \(c>0\), the submanifold \(\{x\in \nu\ :\ \pi(x)\in L,\ |x|=c\}\) is
  an \(\Omega\)-Lagrangian circle-bundle over \(L\).
\end{lem}

\begin{Remark}
  In an earlier version of this paper we implicitly assumed that Lemma
  \ref{lma:biran} holds without the curvature hypothesis, and used
  this to lift Lagrangian torus fibrations from \(Y\) to the
  link. This does not work unless one is able to find connections for
  which the curvature vanishes along all Lagrangians in the torus
  fibration and which satisfy the compatibility conditions
  {\cite[Definitions 2.11 and 2.12]{MTZ}} along the normal crossing
  locus. For example if \(Y\) is a complex curve then the curvature
  condition is empty because Lagrangians in \(Y\) are 1-dimensional.
\end{Remark}

All of this generalises to symplectic submanifolds of higher
codimension whose normal bundle splits as a direct sum of Hermitian
line bundles; the Lagrangian submanifolds from Lemma \ref{lma:biran}
are then torus-bundles.

If we have a simple normal crossing divisor
\(Y=\bigcup_{i\in I}Y_i\subset X\) then an \(\omega\)-regularisation
of \(Y\) is a collection of \(\omega\)-regularisations of the
submanifolds \(Y_J=\bigcap_{j\in J}Y_j\) satisfying compatibility
conditions {\cite[Definitions 2.11 and 2.12]{MTZ}}.

By {\cite[Theorem 2.13]{MTZ}}, after an exact deformation of the
symplectic form \(\omega\) we can find an \(\omega\)-regularisation of
\(Y\). By {\cite[Corollary 5.11]{McLeanReeb}}, this deformation does
not change the contactomorphism type of the link; though McLean's
proof of this is written for ``positively wrapped divisors'' (like
exceptional divisors of resolutions of singularities) it carries over
with minor sign changes to the case of negatively wrapped divisors
(like ample divisors).

Moreover, we can identify the link directly. Let \(r_i\) be a radial
coordinate in the normal bundle \(\nu_i\) of \(Y_i\) (in the sense of
{\cite[Definition 5.7]{McLeanReeb}}). McLean calls a function
\(f\colon X\setminus Y\to\RR\) compatible with \(Y\) if it has the
form \(\sum_im_i\log r_i+\tau\) in a punctured neighbourhood of \(Y\),
for some constants \(m_i\) and a smooth function \(\tau\). The level
sets \(f^{-1}(\epsilon)\) of a compatible function are of contact type
for sufficiently small \(\epsilon\) {\cite[Proposition
  5.8]{McLeanReeb}}. The plurisubharmonic function \(\psi\) on
\(X\setminus Y\) is compatible with \(Y\) {\cite[Lemma
  5.25]{McLeanReeb}}, and so is the function
\(\sum_{i\in I}\log(\GG\circ \mu_i)\) where \(\GG\) is a cut-off
function satisfying
\begin{equation}\label{eq:cutoff2}\GG(x)=\begin{cases}x
&\text{ if }x\in[0,\epsilon/3]\\1&\text{ if
}x\geq\epsilon,\end{cases}\end{equation} and which is positive and
strictly increasing on \([0,\epsilon]\).

\begin{center}
\begin{tikzpicture}[baseline=0cm]
\draw (2.3,0) -- (0,0) -- (0,2.3);
\node at (2.3,0) [right] {\(x\)};
\node at (0,2.3) [left] {\(\GG(x)\)};
\node at (2/3,0) [below] {\(\tfrac{\epsilon}{3}\)};
\node at (2,0) [below] {\(\epsilon\)};
\draw (0,0) -- (2/3,2/3) to[out=45,in=180] (2,2) -- (2.3,2);
\draw[dotted] (2/3,0) -- (2/3,2/3);
\draw[dotted] (2,0) -- (2,2);
\end{tikzpicture}
\end{center}

The link of \(W\) (a level set of \(\psi\) near \(Y\)) is therefore
contactomorphic to a plumbing of the circle bundles with fixed radii
\(r_i\) in the normal bundles of the components \(Y_i\) (a level set
of \(\sum_{i\in I}\log(\GG\circ\mu_i)\)).

\section{Examples}
\label{sct:weinstein}

We now apply Lemma \ref{lma:construction} in some examples to
construct Lagrangian torus fibrations on affine varieties.

\subsection{Example: pair-of-pants}
\label{sct:pop}

Let \(Y\) be a union of four lines \(Y_1,\ldots,Y_4\) in general
position in \(\cp{2}\). The complement \(W:=\cp{2}\setminus Y\) is
called the \emph{4-dimensional pair-of-pants}.

\begin{lem}
  There is a Lagrangian torus fibration \(f\colon\Link(W)\to B\),
  where \(B\) is the 1-skeleton of a tetrahedron:
  
\begin{center}
\begin{tikzpicture}[baseline=0em]
\draw[blue] (0,0) -- (-90:2);
\draw[blue] (0,0) -- (30:2);
\draw[blue] (0,0) -- (150:2);
\draw[blue] (-90:2) -- (-30:1) -- (30:2) -- (90:1) -- (150:2) -- ( -150:1) -- (-90:2);
\filldraw[fill=red,draw=red] (0,0) circle [radius=0.5mm];
\filldraw[fill=red,draw=red] (-90:2) circle [radius=0.5mm];
\filldraw[fill=red,draw=red] (30:2) circle [radius=0.5mm];
\filldraw[fill=red,draw=red] (150:2) circle [radius=0.5mm];
\filldraw[fill=black,draw=black] (-90:1) circle [radius=0.5mm];
\filldraw[fill=black,draw=black] (30:1) circle [radius=0.5mm];
\filldraw[fill=black,draw=black] (150:1) circle [radius=0.5mm];
\filldraw[fill=black,draw=black] (90:1) circle [radius=0.5mm];
\filldraw[fill=black,draw=black] (-30:1) circle [radius=0.5mm];
\filldraw[fill=black,draw=black] (-150:1) circle [radius=0.5mm];
\end{tikzpicture}
\end{center}
\end{lem}
\begin{proof}
  Using {\cite[Theorem 2.13]{MTZ}}, deform the symplectic form and
  choose a symplectic tubular neighbourhood of \(Y\), that is a
  collection of symplectic embeddings
  \(\Psi_i\colon\nbhd_{\mathcal{O}(1)}(Y_i)\to\cp{2}\), where we have
  equipped \(\mathcal{O}(1)\) with the symplectic form from Section
  \ref{sct:sympnbhd}. Write \(N_i=\OP{Image}(\Psi_i)\) for the
  neighbourhood of \(Y_i\), \(N_{ij}=N_i\cap N_j\) for the
  overlaps. 

  We now construct a Lagrangian torus fibration
  \(f\colon\Link(\cp{2}\setminus Y)\to B\) as follows:
  \begin{itemize}
  \item on overlaps \((\mu_i,\mu_j)\colon N_{ij}\to[0,\infty)^2\) is a
    Lagrangian torus fibration. When restricted to
    \(\Link(W)=\{x\ :\ \sum_{i\in I}\log(\GG(\mu_i(x)))=\epsilon'\}\),
    its image is the arc shown below:
    \begin{center}
      \begin{tikzpicture}
        \draw (0,0) -- (2,0);
        \draw (0,0) -- (0,2);
        \draw[blue,thick] (0.3,2) to[out=-90,in=135] (0.7,0.7) to[out=-45,in=180] (2,0.3);
        \node at (0.7,0.7) {\(\bullet\)};
      \end{tikzpicture}
    \end{center}
    (We have drawn the black dot to show how this arc sits inside the
    graph \(B\)). Note that \((\mu_i,\mu_j)\) restricts to function
    \(\varphi_i\colon Y_i\cap N_{ij}\to[0,\infty)\).
  \item on \(N_i^{\circ}:=N_i\setminus\bigcup_{j\neq i}N_{ij}\), the
    projection \(\pi_i\colon \Link(W)\cap N_i^{\circ}\to Y_i\) is a
    circle bundle. Let \(B_i\) be a Y-graph and let
    \(\varphi_i\colon Y_i\to B_i\) be a function extending the
    functions already constructed on \(Y_i\cap N_{ij}\):

    \begin{center}
      \begin{tikzpicture}[baseline=0em]
        \draw (0,0) circle [radius=1cm];
        \filldraw[fill=black] (90:1) circle [radius=0.5mm];
        \filldraw[fill=black] (-30:1) circle [radius=0.5mm];
        \filldraw[fill=black] (-150:1) circle [radius=0.5mm];
        \draw[blue,thick] (180:1) to[out=30,in=45] (240:1);
        \draw[blue,thick] (-60:1) to[out=135,in=150] (0:1);
        \draw[red,thick] (-90:1) to[out=90,in=0] (150:1);
        \draw[red,thick] (-90:1) to[out=90,in=180] (30:1);
        \draw[blue,thick] (60:1) to[out=-90,in=-90] (120:1);
        \draw[blue,dotted,thick] (180:1) to[out=-30,in=120] (240:1);
        \draw[blue,dotted,thick] (-60:1) to[out=90,in=-150] (0:1);
        \draw[red,dotted,thick] (-90:1) to[out=105,in=-45] (150:1);
        \draw[red,dotted,thick] (-90:1) to[out=75,in=-135] (30:1);
        \draw[blue,dotted,thick] (60:1) to[out=190,in=-10] (120:1);
        \draw[->,thick] (1.3,0) -- (2,0);
        \draw[ultra thick,blue,shift={(3,0)}] (0,0) -- (90:1);
        \draw[ultra thick,blue,shift={(3,0)}] (0,0) -- (210:1);
        \draw[ultra thick,blue,shift={(3,0)}] (0,0) -- (-30:1);
        \filldraw[fill=red,draw=red,shift={(3,0)}] (0,0) circle [radius=0.5mm];
        \filldraw[fill=black,draw=black,shift={(3,0)}] (90:1) circle [radius=0.5mm];
        \filldraw[fill=black,draw=black,shift={(3,0)}] (210:1) circle [radius=0.5mm];
        \filldraw[fill=black,draw=black,shift={(3,0)}] (-30:1) circle [radius=0.5mm];
      \end{tikzpicture}
    \end{center}

    The composition
    \(\varphi_i\circ\pi_i\colon\Link(W)\cap N^{\circ}_i\to B_i\) is a
    Lagrangian torus fibration: by Lemma \ref{lma:biran}, the fibres
    are Lagrangian \(S^1\)-bundles over the 1-dimensional fibres of
    \(\varphi_i\) (the vanishing curvature condition is
    trivial). \qedhere
  \end{itemize}
\end{proof}

The cone on the graph \(B\) can be visualised in \(\RR^3\) as the cone
on the 1-skeleton of a tetrahedron. The result of applying
Lemma \ref{lma:construction} in this case is a Lagrangian torus fibration
of \(\cp{2}\setminus Y\) over this cone. The fibres over the cones on
the blue edges are tori; the fibres over the cones on the red vertices
are \(S^1\times 8\) (where \(8\) denotes the wedge of two
circles). The fibre over the cone point is any Lagrangian skeleton for
the pair-of-pants. In particular, since the pair-of-pants is a
hypersurface in \((\CC^*)^3\), Zhou's result \cite{Zhou} implies that
we can take the RSTZ skeleton \cite{RSTZ}. In this particular case,
the RSTZ skeleton is obtained from three disjoint 2-tori by attaching
three cylinders and a triangular 2-cell as indicated in the figure
below. This is homotopy equivalent to the 2-skeleton of a 3-torus.

\begin{center}
\begin{tikzpicture}
\draw[thick,purple,decoration={markings, mark=at position 0.5 with {\arrow{>>}}},postaction={decorate}] (-0.8,-2.2) -- (0,-1.4);
\draw[thick,purple,decoration={markings, mark=at position 0.5 with {\arrow{>>}}},postaction={decorate}] (0,-1.4) -- (0.8,-2.2);
\draw[thick,purple,decoration={markings, mark=at position 0.5 with {\arrow{>}}},postaction={decorate}] (0.8,-2.2) -- (-0.8,-2.2);
\draw[thick] (0,0.3) circle [radius=1cm];
\draw[thick] (-0.3,-0.2) arc [radius=0.5cm,start angle=-90,end angle=90];
\draw[thick] (0.1,0) arc [radius=0.3cm,start angle=-90,end angle=-270];
\draw[thick] (-2,-3) circle [radius=1cm];
\draw[thick] (-2.3,-3.5) arc [radius=0.5cm,start angle=-90,end angle=90];
\draw[thick] (-1.9,-3.3) arc [radius=0.3cm,start angle=-90,end angle=-270];
\draw[thick] (2,-3) circle [radius=1cm];
\draw[thick] (1.7,-3.5) arc [radius=0.5cm,start angle=-90,end angle=90];
\draw[thick] (2.1,-3.3) arc [radius=0.3cm,start angle=-90,end angle=-270];
\draw[thick,purple,decoration={markings, mark=at position 0.5 with {\arrow{>}}},postaction={decorate}] (0.5,-2.5) -- (-0.5,-2.5);
\draw[thick] (-0.5,-3.5) -- (0.5,-3.5);
\draw[thick,red,decoration={markings, mark=at position 0.125 with {\arrow{>}}},postaction={decorate}] (-0.5,-3) circle [x radius=0.3cm,y radius=0.5cm];
\draw[thick,blue,decoration={markings, mark=at position 0.75 with {\arrow{>}}},postaction={decorate}] (0.5,-3.5) arc [x radius=0.3cm,y radius=0.5cm,start angle=-90,end angle=90];
\draw[thick,blue,dotted] (0.5,-3.5) arc [x radius=0.3cm,y radius=0.5cm,start angle=-90,end angle=-270];
\draw[thick,red,decoration={markings, mark=at position 0.25 with {\arrow{>}}},postaction={decorate}] (-1,-3) arc [x radius=0.4cm,y radius=0.2cm,start angle=0,end angle=-180];
\draw[thick,red,dotted] (-1,-3) arc [x radius=0.4cm,y radius=0.2cm,start angle=0,end angle=180];
\draw[thick,blue,decoration={markings, mark=at position 0.125 with {\arrow{>}}},postaction={decorate}] (2,-3) circle [x radius=0.3cm,y radius=0.5cm];
\draw[thick,blue,decoration={markings, mark=at position 0.125 with {\arrow{>>}}},postaction={decorate}] (-2,-3) circle [x radius=0.3cm,y radius=0.5cm];
\draw[thick,blue,decoration={markings, mark=at position 0.125 with {\arrow{>>}}},postaction={decorate}] (0,0.3) circle [x radius=0.3cm,y radius=0.5cm];
\draw[thick,red,decoration={markings, mark=at position 0.75 with {\arrow{>>}}},postaction={decorate}] (2,-2.7) arc [x radius=0.2cm,y radius=0.35cm,start angle=-90,end angle=-270];
\draw[thick,red,dotted] (2,-2.7) arc [x radius=0.2cm,y radius=0.35cm,start angle=-90,end angle=90];
\draw[thick,red,decoration={markings, mark=at position 0.25 with {\arrow{>>}}},postaction={decorate}] (0,-0.7) arc [x radius=0.2cm,y radius=0.35cm,start angle=-90,end angle=-270];
\draw[thick,red,dotted] (0,-0.7) arc [x radius=0.2cm,y radius=0.35cm,start angle=-90,end angle=90];
\begin{scope}[shift={(-3.2,0.8)},rotate=45]
\draw (-0.5,-2.5) -- (0.5,-2.5);
\draw[thick,purple,decoration={markings, mark=at position 0.5 with {\arrow{>>}}},postaction={decorate}] (-0.5,-3.5) -- (0.5,-3.5);
\draw[thick,blue,decoration={markings, mark=at position 0.125 with {\arrow{>>}}},postaction={decorate}] (-0.5,-3) circle [x radius=0.3cm,y radius=0.5cm];
\draw[thick,red,decoration={markings, mark=at position 0.75 with {\arrow{>>}}},postaction={decorate}] (0.5,-3.5) arc [x radius=0.3cm,y radius=0.5cm,start angle=-90,end angle=90];
\draw[thick,red,dotted] (0.5,-3.5) arc [x radius=0.3cm,y radius=0.5cm,start angle=-90,end angle=-270];
\end{scope}
\begin{scope}[shift={(3.2,0.8)},rotate=-45]
\draw (-0.5,-2.5) -- (0.5,-2.5);
\draw[thick,purple,decoration={markings, mark=at position 0.5 with {\arrow{>>}}},postaction={decorate}] (-0.5,-3.5) -- (0.5,-3.5);
\draw[thick,blue,decoration={markings, mark=at position 0.125 with {\arrow{>>}}},postaction={decorate}] (-0.5,-3) circle [x radius=0.3cm,y radius=0.5cm];
\draw[thick,red,decoration={markings, mark=at position 0.75 with {\arrow{>>}}},postaction={decorate}] (0.5,-3.5) arc [x radius=0.3cm,y radius=0.5cm,start angle=-90,end angle=90];
\draw[thick,red,dotted] (0.5,-3.5) arc [x radius=0.3cm,y radius=0.5cm,start angle=-90,end angle=-270];
\end{scope}
\end{tikzpicture}
\end{center}

\begin{Remark}
In this case, Mikhalkin gave a purely topological torus fibration of
the pair-of-pants over the same base, namely the tropicalisation map
\cite{Mikhalkin}. See also the paper of Golla and Martelli
\cite{GollaMartelli} for a description of the Mikhalkin
fibration. The fact that the 4-dimensional pair-of-pants is homotopy
equivalent to the 2-skeleton of a 3-torus was proved much earlier by
Salvetti \cite{Salvetti}.

\end{Remark}
\subsection{Example: negative vertex}
\label{sct:negvtx}

Take the affine variety
\[W=\{(x,y,u_1,u_2)\in\CC^2\times(\CC^*)^2\ :\ xy=u_1+u_2+1\}.\] We
call this variety {\em the negative vertex}. It can be compactified to
\(\cp{3}\) as follows. Rewrite the equation for \(W\) as
\(u_2=xy-u_1-1\); \(W\) is the complement of \(u_2=0\) in
\(\CC_{xy}^2\times\CC_{u_1}^*\), in other words it is the complement
in \(\cp{3}_{[x:y:u_1:w]}\) of \((xy-u_1w-w^2)u_1w=0\). Taking
\(L_1=\mathcal{O}(2)\), \(L_2=L_3=\mathcal{O}(1)\),
\(s_1=xy-u_1w-w^2\), \(s_2=u_1\), \(s_3=w\), we obtain a subvariety
\(Y=Y_1\cup Y_2\cup Y_3\) such that \(W=\cp{3}\setminus Y\). However,
this subvariety is not normal crossing: the components \(Y_1\) and
\(Y_3\) intersect non-transversely at \([0:0:1:0]\) (\(Y_3\) is the
tangent plane to the quadric \(Y_1\) at that point).

If we blow up along the line \([x:0:u_1:0]\) then the total transform
of \(Y\) becomes simple normal crossing. We continue to write \(Y_j\)
for the proper transform of \(Y_j\), \(j=1,2,3\), and we write \(Y_4\)
for the exceptional divisor (a copy of \(\cp{1}\times\cp{1}\)). We
will consider the ample divisor \(Y_1+Y_2+Y_3+Y_4\) (i.e. giving all
components multiplicity one). The symplectic form coming from this
ample divisor is in the cohomology class \(4H-E\), where \(H\) is the
cohomology class Poincar\'{e} dual to the proper transform of a
generic plane and \(E\) is the class of the exceptional divisor.

Without choosing very specific curvature forms on the normal bundles,
we no longer have access to Lemma \ref{lma:biran} for constructing
Lagrangian torus fibrations on the link. Instead, we will proceed in a
more {\em ad hoc} manner and construct the torus fibration directly
using almost toric methods.

In Figure \ref{fig:negvert}, we draw almost toric base diagrams for a
system of Lagrangian torus fibrations \(\varphi_j\colon Y_j\to
B_j\). Recall that crosses indicate focus-focus singularities and
dotted lines indicate branch cuts. Some of the edges are broken by a
branch cut: they are nonetheless straight lines in the affine
structure and the break point is not to be considered a vertex. The
colours on edges indicate how these edges are to be identified in the
pushout. Broken edges are decorated with the same colour on each
segment, which is taken to mean that the edge continues beyond the
break point, not that some kind of self-identification should be
made. Dots on the edges are there to indicate interior integral points
of the edges (for the integral affine structure). You can read the
affine lengths of edges off from the integrals of \([\omega]=4H-E\)
over the corresponding spheres.

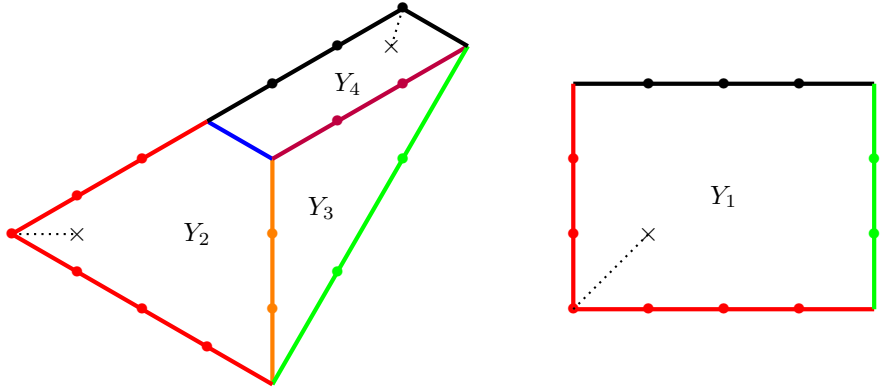
\begin{figure}
\begin{center}
\begin{tikzpicture}
\draw[orange,ultra thick] (0,0) -- (-90:3);
\node[orange] at (-90:1) {\(\bullet\)};
\node[orange] at (-90:2) {\(\bullet\)};
\draw[blue,ultra thick] (0,0) -- (150:1);
\draw[purple,ultra thick] (0,0) -- (30:3);
\node[purple] at (30:1) {\(\bullet\)};
\node[purple] at (30:2) {\(\bullet\)};
\draw[green,ultra thick] (30:3) -- (-90:3);
\node[green,shift={(0,-3)}] at ({sqrt(3)/2},3/2) {\(\bullet\)};
\node[green,shift={(0,-3)}] at ({2*sqrt(3)/2},3) {\(\bullet\)};
\draw[black,ultra thick] (30:3) -- ({sqrt(3)},2) -- (150:1);
\node[black] at (0,1) {\(\bullet\)};
\node[black] at ({sqrt(3)/2},1.5) {\(\bullet\)};
\node[black] at ({sqrt(3)},2) {\(\bullet\)};
\draw[dotted,thick] ({sqrt(3)},2) -- ({sqrt(3)-0.15},2-0.5) node {\(\times\)};
\draw[red,ultra thick,shift={(0,-3)}] (0,0) -- (150:4);
\draw[red,ultra thick] (150:1) -- ({-2*sqrt(3)},-1);
\draw[dotted,thick] ({-2*sqrt(3)},-1) -- ({-1.5*sqrt(3)},-1) node {\(\times\)};
\node[red,shift={(0,-3)}] at (150:1) {\(\bullet\)};
\node[red,shift={(0,-3)}] at (150:2) {\(\bullet\)};
\node[red,shift={(0,-3)}] at (150:3) {\(\bullet\)};
\node[red,shift={(0,-3)}] at (150:4) {\(\bullet\)};
\node[red] at ({-1.5*sqrt(3)},-0.5) {\(\bullet\)};
\node[red] at (-{sqrt(3)},0) {\(\bullet\)};
\draw[black,ultra thick,shift={(6,0)}] (2,1) -- (-2,1);
\node[black,shift={(6,0)}] at (1,1) {\(\bullet\)};
\node[black,shift={(6,0)}] at (0,1) {\(\bullet\)};
\node[black,shift={(6,0)}] at (-1,1) {\(\bullet\)};
\draw[red,ultra thick,shift={(6,0)}] (-2,1) -- (-2,-2) -- (2,-2);
\node[red,shift={(6,0)}] at (-2,0) {\(\bullet\)};
\node[red,shift={(6,0)}] at (-2,-1) {\(\bullet\)};
\node[red,shift={(6,0)}] at (-2,-2) {\(\bullet\)};
\node[red,shift={(6,0)}] at (0,-2) {\(\bullet\)};
\node[red,shift={(6,0)}] at (-1,-2) {\(\bullet\)};
\node[red,shift={(6,0)}] at (1,-2) {\(\bullet\)};
\draw[green,ultra thick,shift={(6,0)}] (2,-2) -- (2,1);
\node[green,shift={(6,0)}] at (2,0) {\(\bullet\)};
\node[green,shift={(6,0)}] at (2,-1) {\(\bullet\)};
\draw[dotted,thick,black,shift={(6,0)}] (-2,-2) -- (-1,-1) node {\(\times\)};
\node at (-1,-1) {\(Y_2\)};
\node at (1,1) {\(Y_4\)};
\node at (0.65,-0.65) {\(Y_3\)};
\node at (6,-0.5) {\(Y_1\)};
\end{tikzpicture}
\end{center}
\caption{Almost toric base diagrams for a system of
  Lagrangian torus fibrations on the boundary divisor for the negative
  vertex.}
\label{fig:negvert}
\end{figure}

\begin{Proposition}\label{prp:neg1}
  We can use this system of fibrations on \(Y\) to construct a
  Lagrangian torus fibration on the link. The base \(B\) is a
  topological 2-sphere and the discriminant locus consists of three
  points (the focus-focus singularities of the almost toric
  fibrations).  We further obtain a Lagrangian torus fibration on
  \(W\); the base is a 3-ball and the discriminant locus is the cone
  over three points, i.e. a Y-graph.
\end{Proposition}
\begin{proof}
  We will prove this in Section \ref{sct:prf_neg1} below.
\end{proof}

\begin{thm}\label{thm:negvtx}
The negative vertex admits a Lagrangian torus fibration over the
3-ball such that the discriminant locus is a Y-graph. The fibre over
the vertex of the Y-graph is a topological space obtained by
attaching a solid torus to a wedge of two circles via an attaching
map which is freely homotopic to a map \(\phi\colon
\partial(D^2\times S^1)\to S^1\vee S^1\) which induces the map
\(\phi_*\colon\ZZ^2\to\ZZ\star\ZZ=\langle a,b\rangle\),
\(\phi_*(1,0)=aba^{-1}b^{-1}\), \(\phi_*(0,1)=1\) on fundamental
groups.
\end{thm}
\begin{proof}
  The existence of the Lagrangian torus fibration follows immediately
  from Lemma \ref{lma:construction} and Proposition \ref{prp:neg1}. It
  remains only to find the Lagrangian skeleton of \(W\), which we do
  in Section \ref{prf:thm:negvtx} below.
\end{proof}
\begin{Remark}\label{rmk:negvtx}
Note that since \(S^1\vee S^1\) is an Eilenberg-MacLane space, the
induced map on fundamental groups determines the attaching map up to
homotopy.

\end{Remark}
\begin{Remark}\label{rmk:grossfibre}
  Note that the Gross fibre (cf \cite[Example 2.6.(4)]{Gross}) 
  is also obtained by
  attaching a solid torus to a wedge of two circles by an attaching
  map in this homotopy class. It seems harder to see the relationship
  between the RSTZ skeleton in this case and the Gross fibre, though
  they are necessarily homotopy equivalent.
\end{Remark}

\subsubsection{Proof of Proposition \ref{prp:neg1}}
\label{sct:prf_neg1}

We use {\cite[Theorem 2.13]{MTZ}} to make an exact deformation of the
symplectic structure so that \(Y\) admits an
\(\omega\)-regularisation. Now by Lemma \ref{lma:app1} and Remark
\ref{rmk:plumbingapp}, a neighbourhood of \(Y\) in \(X\) can be
obtained from the normal bundles \(\nu_XY_i\) by plumbing along the
submanifolds \(\nu_XY_I\) (\(|I|\geq 2\)), and any two plumbing
neighbourhoods are related by gauge transformations of \(\nu_XY_I\)
preserving the stratification by subbundles \(\nu_{Y_J}Y_I\). Isotopic
gauge transformations yield symplectomorphic plumbings, so we can work
with unitary gauge transformations without loss of generality (because
\(Sp(2n)\) retracts onto its maximal compact \(U(n)\)).

Momentarily, we will write down an almost toric base diagram for a
6-manifold \(\mathcal{X}\) whose almost toric boundary \(\mathcal{Y}\)
is precisely the system of almost toric 4-manifolds shown in Figure
\ref{fig:negvert}, which is our divisor \(Y\). The normal bundles to
the components and strata of \(\mathcal{Y}\) agree with those of
\(Y\).

The submanifolds \(\nu_{\mathcal{X}}\mathcal{Y}_I\) for \(|I|\geq 2\)
are toric: they are total spaces of
\(\mathcal{O}(m)\oplus\mathcal{O}(n)\to\cp{1}\) for some \(m,n\). If
\(g\) is a unitary gauge transformation of this bundle preserving the
stratification by subbundles \(\mathcal{O}(m)\) and
\(\mathcal{O}(n)\), then it preserves the toric fibration: each toric
fibre is invariant under the group of rotations \(U(1)\times U(1)\) of
the bundle fibre. Therefore all possible plumbings admit almost toric
structures with the same base diagram\footnote{Recall that an almost
  toric base diagram determines the almost toric manifold only once
  certain characteristic classes are fixed \cite{Zung}.}. In
particular, this implies that \(X\) (which corresponds to some choice
of plumbing) admits an almost toric fibration. Restricting this to the
link of \(Y\) gives the desired Lagrangian torus fibration. For
example, you can take the link to be the preimage of a hypersurface in
the almost toric base diagram.

We now describe the almost toric base diagram for \(\mathcal{X}\).
Let
\begin{align*}
  P &= (0,0,0), & Q &= (0,0,3) \\
  B_{134} &= (1,0,0), & B_{123} &= (4,0,3) \\
  B_{124} &= (0,3,3), & B_{234} &= (1,3,3),
\end{align*}
and consider the convex hull of these six points.

\begin{center}
  \includegraphics[width=350pt]{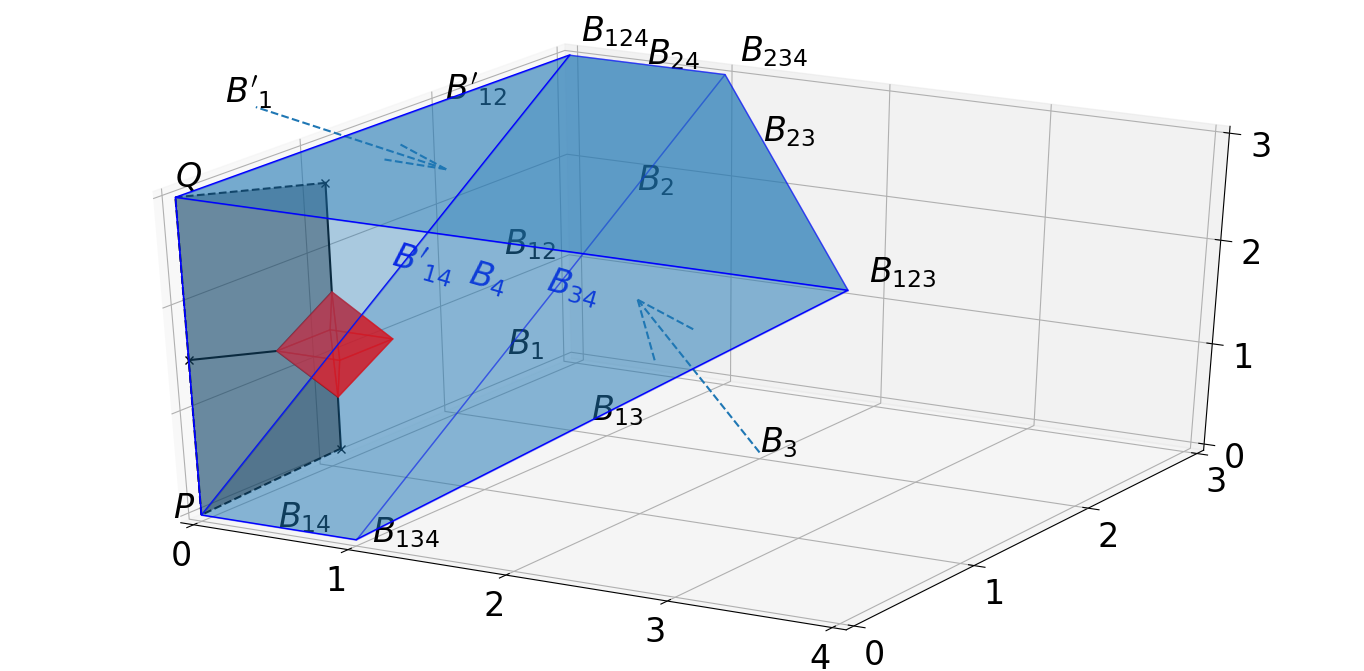}
\end{center}

We name the facets:
\begin{itemize}
\item \(B_1 = \langle P, B_{134}, B_{123}, Q\rangle\)
  (the quadrilateral at the front in the figure),
\item \(B'_1 = \langle P, Q, B_{124}\rangle\) (the triangle at the
  back on the left in the figure),
\item \(B_2 = \langle B_{124}, Q, B_{123}, B_{234}\rangle\) (the
  quadrilateral on the top in the figure),
\item \(B_3 = \langle B_{134}, B_{123}, B_{234}\rangle\) (the slanted
  triangle on the bottom right in the figure),
\item \(B_4 = \langle B_{124}, B_{234}, B_{134}, P\rangle\) (the
  slanted rectangle at the back underneath in the figure).
\end{itemize}
We label the edges \(B_{ij}:=B_i\cap B_j\) and
\(B'_{ij}:=B'_i\cap B_j\).

We will excise a small neighbourhood of \((0.5, 0.5, 1.5)\) (red in
the figure) and call the resulting polytope \(\mathbf{B}\). We will
decorate \(\mathbf{B}\) with the data of a 3-dimensional almost toric
base diagram:
\begin{itemize}
\item We make a branch cut using a half-plane emanating from the vertical
  line at \(x=y=0.5\) (the vertical black line in the figure) and
  containing the direction \((-1,-1,0)\). We call this the {\em branch
  plane}.
\item We insert three rays of focus-focus singularities, emanating:
  \begin{itemize}
  \item vertically downwards from \((0.5,0.5,3)\),
  \item vertically upwards from \((0.5,0.5,0.5)\),
  \item horizontally in the \((1,1,0)\)-direction from \((0,0,1.5)\).
  \end{itemize}
  These three rays disappear into the excised neighbourhood. The
  horizontal ray cuts the branch plane into an upper and lower region.
\item We reglue the affine structures across the branch plane as follows:
  \begin{itemize}
  \item When crossing from \(y<x\) to \(y>x\) across the upper region
    of the branch plane, we use the matrix
    \(M_1=\begin{pmatrix} 2 & -1 & 0 \\ 1 & 0 & 0 \\ 0 & 0 &
      1\end{pmatrix}\).
  \item When crossing from \(y<x\) to \(y>x\) across the lower region
    of the branch plane, we use the matrix
    \(M_2=\begin{pmatrix} 2 & -1 & 0 \\ 1 & 0 & 0 \\ 1 & -1 &
      1\end{pmatrix}\).
  \end{itemize}
\end{itemize}
\begin{Definition}
  Let \(\mathcal{X}\) be an almost toric 6-manifold associated to this
  almost toric base diagram, and \(\mathcal{Y}_i\) the almost toric
  4-manifold associated to the facet \(B_i\). Note that the base for
  each \(\mathcal{Y}_i\) is contractible, hence determines
  \(\mathcal{Y}_i\) completely. On the other hand, the base for
  \(\mathcal{X}\) has nontrivial second cohomology, so can arise as
  the base diagram for different almost toric 6-manifolds
  (distinguished by Zung's Chern class \cite{Zung}).
\end{Definition}
\begin{Remark}
  The union \(\mathcal{Y}=\bigcup_{i=1}^4\mathcal{Y}_i\) is a
  symplectic simple normal crossing divisor with pairwise
  symplectically orthogonal intersections.
\end{Remark}
With this integral affine structure, the facets have the following
properties:
\begin{itemize}
\item \(B_2\) is the almost toric base diagram for an almost toric
  structure on the first Hirzebruch surface (obtained from the
  standard moment quadrilateral by a nodal trade). In particular, the
  edges \(B'_{12}\) and \(B_{12}\) form part of a straight line in the
  reglued integral affine structure (this is because the tangent
  vectors to these edges are related by \(M_1\begin{pmatrix} 0 \\ -1
    \\ 0\end{pmatrix}=\begin{pmatrix} 1 \\ 0 \\ 0\end{pmatrix}\)).
\item \(B_3\) is the standard toric moment polytope for \(\cp{2}\).
\item \(B_4\) is the almost toric base diagram for an almost toric
  structure on \(\cp{1}\times\cp{1}\) obtained from a standard moment
  rectangle by a nodal trade. In particular, the edges \(B'_{14}\) and
  \(B_{14}\) form part of a straight line in the reglued integral
  affine structure (this is because the tangent vectors to these edges
  are related by \(M_2\begin{pmatrix} 0 \\ -1
  \\ -1\end{pmatrix}=\begin{pmatrix} 1 \\ 0 \\ 0\end{pmatrix}\)).
\item \(B_1\cup B'_1\) forms a single facet in the reglued integral affine
  structure. This corresponds to an almost toric structure on
  \(\cp{1}\times\cp{1}\) with one nodal fibre, shown in the diagram
  below:

  \begin{center}
    \begin{tikzpicture}
      \filldraw[thick,fill=blue,opacity=0.4] (0,0) -- (1,0) -- (4,3) -- (-3,3) -- cycle;
      \draw[thick,dashed] (0,0) -- (0,1) node {\(\times\)};
      \node at (0,0) {\(\bullet\)};
      \node at (1,0) {\(\bullet\)};
      \node at (2,1) {\(\bullet\)};
      \node at (3,2) {\(\bullet\)};
      \node at (4,3) {\(\bullet\)};
      \node at (3,3) {\(\bullet\)};
      \node at (2,3) {\(\bullet\)};
      \node at (1,3) {\(\bullet\)};
      \node at (0,3) {\(\bullet\)};
      \node at (-1,3) {\(\bullet\)};
      \node at (-2,3) {\(\bullet\)};
      \node at (-3,3) {\(\bullet\)};
      \node at (-2,2) {\(\bullet\)};
      \node at (-1,1) {\(\bullet\)};
      \node at (3,1.5) {\(B_{13}\)};
      \node at (-2,1.5) {\(B'_{14}\)};
      \node at (0.5,-0.5) {\(B_{14}\)};
      \node at (0.5,3.5) {\(B_{12}\cup B'_{12}\)};
    \end{tikzpicture}
  \end{center}
\end{itemize}  
By a mutation and a shear, this is equivalent to:

\begin{center}
  \begin{tikzpicture}
    \filldraw[thick,fill=blue,opacity=0.4] (0,0) -- (4,0) -- (4,3) -- (0,3) -- cycle;
    \draw[dashed] (0,3) -- (2,1) node {\(\times\)};
    \node at (0,0) {\(\bullet\)};
    \node at (1,0) {\(\bullet\)};
    \node at (2,0) {\(\bullet\)};
    \node at (3,0) {\(\bullet\)};
    \node at (4,0) {\(\bullet\)};
    \node at (4,1) {\(\bullet\)};
    \node at (4,2) {\(\bullet\)};
    \node at (4,3) {\(\bullet\)};
    \node at (3,3) {\(\bullet\)};
    \node at (2,3) {\(\bullet\)};
    \node at (1,3) {\(\bullet\)};
    \node at (0,3) {\(\bullet\)};
    \node at (0,2) {\(\bullet\)};
    \node at (0,1) {\(\bullet\)};
    \node at (4,1.5) [right] {\(B_{13}\)};
    \node at (2,0) [below] {\(B'_{14}\cup B_{14}\)};
    \node at (2,3) [above] {\(B_{12}\)};
    \node at (0,1.5) [left] {\(B'_{12}\)};
  \end{tikzpicture}
\end{center}

We now treat each \(B'_{ij}\cup B_{ij}\) as a single edge and write it
as \(B_{ij}\). We write \(\mathcal{Y}_{ij}\) for the corresponding
symplectic 2-manifold in \(\mathcal{X}\).
\begin{itemize}
\item \(\mathcal{Y}_{13}\) is a copy of
  \(\{pt\}\times\cp{1}\subset \mathcal{Y}_1\cong\cp{1}\times\cp{1}\)
  and a line in \(\mathcal{Y}_3\cong \cp{2}\).
\item \(\mathcal{Y}_{14}\) is a copy of
  \(\cp{1}\times\{pt\}\subset \mathcal{Y}_1\) and a \((1,1)\)-curve in
  \(\mathcal{Y}_4\cong\cp{1}\times\cp{1}\).
\item \(\mathcal{Y}_{34}\) is a line in \(\mathcal{Y}_3\) and a copy
  of \(\{pt\}\times \cp{1}\) in \(\mathcal{Y}_4\).
\item \(\mathcal{Y}_{24}\) is a copy of \(\cp{1}\times\{pt\}\) in
  \(\mathcal{Y}_4\) and the \(-1\)-curve in
  \(\mathcal{Y}_2\cong\mathbb{F}_1\).
\item \(\mathcal{Y}_{23}\) is a line in \(\mathcal{Y}_3\) and a fibre
  of \(\mathbb{F}_1\to\cp{1}\) in \(\mathcal{Y}_1\).
\item \(\mathcal{Y}_{12}\) is a \((1,1)\)-curve in \(\mathcal{Y}_1\)
  and a \((1,1)\) curve in \(\mathcal{Y}_2\) (with respect to the
  basis given by a fibre and the section with square \(1\)).
\end{itemize}
Note that from this you can read off the normal bundles of the
components \(\mathcal{Y}_i\); for example, \(\nu\mathcal{Y}_3\) is
determined by its restriction to
\(\nu_{\mathcal{Y}_1}\mathcal{Y}_{13}\) along \(\mathcal{Y}_{13}\),
which is trivial because \(\{pt\}\times\cp{1}\) has trivial normal
bundle in \(\cp{1}\times\cp{1}\). The Poincar\'{e} duals of the normal
bundles are:
\begin{align*}
c_1(\nu \mathcal{Y}_1) &= [\mathcal{Y}_{13}] + [\mathcal{Y}_{14}], & c_1(\nu \mathcal{Y}_2) &= [\mathcal{Y}_{24}] + [\mathcal{Y}_{23}]\\
c_1(\nu \mathcal{Y}_3) &= 0 & c_1(\nu \mathcal{Y}_4) &= [\mathcal{Y}_{24}] - [\mathcal{Y}_{34}].
\end{align*}

This completes the construction of the base diagram for
\(\mathcal{X}\) and the identification of its almost toric boundary
\(\mathcal{Y}\) with \(Y\).

\subsubsection{Proof of Theorem \ref{thm:negvtx}}
\label{prf:thm:negvtx}

It remains to identify the Lagrangian skeleton of the negative
vertex. The theorem then follows from Lemma \ref{lma:negvtx1} and
Lemma \ref{lma:negvtx2} below:
\begin{itemize}
\item Lemma \ref{lma:negvtx1} identifies the critical locus of a
  suitably-chosen plurisubharmonic function \(\psi\). It shows that
  there are three isolated critical points inside the locus \(xy=0\)
  (whose downward manifolds trace out a copy of \(S^1\vee S^1\)) and a
  circle of critical points with \(xy\neq 0\) (whose downward manifold
  is a solid torus).
\item Lemma \ref{lma:negvtx2} identifies the attaching map for the
  solid torus to \(S^1\vee S^1\).
\end{itemize}

\begin{lem}\label{lma:negvtx1}
If \((x,y,u_1,u_2)\in W\), we will write \(x=Xe^{i\alpha}\),
\(y=Ye^{i\beta}\), \(u_j=e^{R_j+i\theta_j}\). Let
\(\Phi=\alpha+\beta\). Fix a real constant \(c\in(-1,-\ln 2)\) and
consider the plurisubharmonic function
\[\psi=\frac{1}{2}\left(X^2+Y^2+\sum (R_j-c)^2\right).\] The
critical locus of \(\psi|_W\) comprises:
\begin{itemize}
\item the points \[P_1=\left(0,0,-\frac{1}{2},-\frac{1}{2}\right),\quad
P_2=(0,0,e^{a(c)},-e^{b(c)}),\quad
P_3=(0,0,-e^{b(c)},e^{a(c)}),\] where \((a(c),b(c))\) is the
unique point of intersection between the curves
\(e^{R_1}+1=e^{R_2}\) and
\((R_1-c)e^{-(R_1-c)}=-(R_2-c)e^{-(R_2-c)}\);
\item the circle of points
\[\left(e^{i\alpha}\sqrt{2e^{R(c)}+1},e^{-i\alpha}\sqrt{2e^{R(c)}+1},-e^{R(c)},-e^{R(c)}\right),\]
where \(R(c)=c-\mathcal{W}(e^c)\), and \(\mathcal{W}\) is the {\em Lambert
\(W\)-function}\footnote{The Lambert \(W\)-function is the inverse
to the function \(x\mapsto xe^x\) for \(x>0\).}.
\end{itemize}
The downward manifolds of the critical points \(P_2,P_3\) trace out
a figure 8 with vertex at \(P_1\). The downward manifold of the
circle of critical points is a solid torus.
\end{lem}
\begin{proof}
We need to find the critical points of the constrained functional
\begin{gather*}
\frac{1}{2}\left(X^2+Y^2+\sum(R_j-c)^2\right)\\
\hspace{2cm}-\lambda\left(XY\cos\Phi-1-\sum e^{R_j}\cos\theta_j\right)\\
\hspace{4cm}-\mu\left(XY\sin\Phi-\sum
e^{R_j}\sin\theta_j\right).
\end{gather*}
where \(\lambda\) and \(\mu\) are Lagrange multipliers imposing the
constraint \(xy=1+\sum u_j\). Differentiating, the critical point
equations are
\begin{align}
\label{eq:cp1}X&=Y(\lambda \cos\Phi-\mu\sin\Phi)\\
\label{eq:cp2}Y&=X(\lambda \cos\Phi-\mu\sin\Phi)\\
\label{eq:cp3}0&=XY(\lambda\sin\Phi-\mu\cos\Phi)\\
\label{eq:cp4}R_j-c&=-e^{R_j}(\lambda\cos\theta_j+\mu\sin\theta_j)\\
\label{eq:cp5}\lambda\sin\theta_j&=\mu\cos\theta_j.
\end{align}
We identify two cases: \(XY=0\) and \(XY\neq 0\).

{\bf Case \(XY\neq 0\).} If \(X\) and \(Y\) are nonzero then Equation
\eqref{eq:cp3} implies \((\lambda,\mu)=k(\cos\Phi,\sin\Phi)\) for some
\(k\neq 0\). Equation \eqref{eq:cp5} implies that, for all \(j\),
\((\lambda,\mu)=k_j(\cos\theta_j,\sin\theta_j)\) for some \(k_j\).
Overall, this means that the angles \(\theta_1,\ldots,\theta_n\) and
\(\Phi\) agree modulo \(\pi\) and that
\(k_je^{i\theta_j}=ke^{i\Phi}\), so \(k/k_j=\pm 1\). The constraint
equation becomes \(\left(XY-\sum (ke^{R_i}/k_j)\right)e^{i\Phi}=1\),
which implies that \(\Phi\) is either \(0\) or \(\pi\). Equations
\eqref{eq:cp1} and \eqref{eq:cp2} tell us that \(X=kY\) and \(Y=kX\),
so \(X=k^2X\) and \(k=\pm 1\). Indeed, since both \(X\) and \(Y\) are
positive, this implies \(k=1\) and \(X=Y\). Since \(|k_j|=|k|=1\),
this implies that \(k_j=\pm 1\).

Equation \eqref{eq:cp4} becomes \(R_j-c=-k_je^{R_j}\), or
\[(R_j-c)e^{-(R_j-c)}=-k_je^c.\] If \(k_j<0\) then, provided
\(c>-1\), this has no solutions (as \(xe^{-x}\leq e^{-1}\) for all
\(x\)). So if we choose \(c>-1\) we need to take \(k_j=1\). In this
case, there is a unique solution\footnote{In fact, \(R(c)=c-\mathcal{W}(e^c)\)
where \(\mathcal{W}\) is the Lambert \(W\)-function.} \(R(c)\) to the equation
\((R-c)e^{-(R-c)}=-e^c\). This gives a circle of solutions \[
X=Y=\sqrt{e^{i\Phi}+2e^{R(c)}},\quad R_j=R(c),\quad \theta_j=\Phi,
\] (as \(\alpha\) and \(\beta\) vary subject to
\(\alpha+\beta=\Phi\)) provided \(e^{i\Phi}+2e^{R(c)}\geq 0\).

If \(e^{i\Phi}=1\) then \(e^{i\Phi}+2e^{R(c)}>0\), giving a circle
of solutions. If \(e^{i\Phi}=-1\) then we can ensure there are no
solutions by taking \(c\) sufficiently small. Indeed, if \(c<-0.3\)
then \(\ln 2<\mathcal{W}(e^c)-c\) so \(-1+2e^{R(c)}<0\).

{\bf Case \(XY=0\):} If one of \(X\) or \(Y\) vanishes, then so does
the other (using Equations \eqref{eq:cp1} and \eqref{eq:cp2}). One
could find the critical points from the statement of the lemma by a
detailed computation as in the other case, but there is a nice
``picture-proof'' in this case. The subset of points in \(W\) with
\(x=y=0\) is the curve \(C=\{(u_1,u_2)\in(\CC^*)^2\ :\ u_1+u_2+1=0\}\)
in \((\CC^*)^2\). If we draw the image of this curve under the map
\((R_1,R_2)\colon W\to\RR^2\) then we get the amoeba shown below. We
also show the level sets of \(\psi|_C\) as dotted circles, and it is
easy to see these critical points (blue dots). In red, we have given
an idea of how the gradient flowlines from these critical points look:
the downward manifold from each of \(P_2\) and \(P_3\) is an interval,
whose boundary points must tend to the index 0 critical point \(P_1\)
in the limit. Note that it is not true in general that one can find
the critical points and flowlines by restricting to a submanifold in
this way, but in this case, the Hessian is positive definite on the
normal directions to \(C\). \qedhere

\begin{center}
\begin{tikzpicture}[scale=1.2]
\draw (-2,0) -- (2,0);
\draw (0,-2) -- (0,2);
\draw[thick] (-2,-0.1) to[out=0,in=90] (-0.1,-2);
\draw[thick] (-2,0.1) to[out=0,in=-135] (1.9,2);
\draw[thick] (0.1,-2) to[out=90,in=-135] (2,1.9);
\draw[dashed] (-1,-1) circle [radius=0.48];
\draw[dashed] (-1,-1) circle [radius=1.18];
\filldraw (-1,-1) circle [radius=0.02];
\filldraw[blue] (-135:0.915) circle [radius=0.05];
\node[blue] at (-135:0.915) [above right] {\(P_1\)};
\filldraw[blue] (-1.2,0.16) circle [radius=0.05];
\node[blue] at (-1.2,0.16) [above] {\(P_3\)};
\filldraw[blue] (0.16,-1.2) circle [radius=0.05];
\node[blue] at (0.16,-1.2) [right]{\(P_2\)};
\draw[thick,red] (-1.2,0.16) to[out=0,in=110] (-135:0.915);
\draw[thick,red] (0.16,-1.2) to[out=90,in=0] (-135:0.915);
\draw[dotted,thick,red] (-1.2,0.16) to[out=-90,in=130] (-135:0.915);
\draw[dotted,thick,red] (0.16,-1.2) to[out=180,in=-20] (-135:0.915);
\node (A) at (-3.2,-0.7) {\((-\ln 2,-\ln 2)\)};
\node (B) at (-135:0.915) {};
\draw[->] (A) -- (B);
\node (C) at (-3,-1.7) {\((c,c)\)};
\draw[->] (C.east) to[bend right] (-1.05,-1.05);
\end{tikzpicture}
\end{center}
\end{proof}

\begin{lem}\label{lma:zariski}
The fundamental group \(\pi_1(W)\) is abelian.
\end{lem}
\begin{proof}
There is an affine conic fibration \(\pi\colon W\to(\CC^*)^2\),
\(\pi(x,y,u_1,u_2)=(u_1,u_2)\), with singular fibres over the curve
\(u_1+u_2+1=0\). Let \(U\subset(\CC^*)^2\) be the complement of this
curve and \(V=\pi^{-1}(U)\). Since \(V\) is a Zariski open set in
\(W\), the inclusion map \(i\colon V\to W\) induces a surjection
\(i_*\colon\pi_1(V)\to\pi_1(W)\).

Note that \(U\) is a 4-dimensional pair-of-pants; \(U\) therefore
deformation retracts onto the 2-skeleton of a 3-torus, so
\(\pi_1(U)=\ZZ^3\). Let \(\gamma\) be a circle in a smooth fibre of
\(\pi\) onto which the fibre deformation retracts. The fundamental
group of \(V\) is a central extension of \(\pi_1(U)\) by
\(\ZZ\langle\gamma\rangle\). We have \(i_*(\gamma)=0\) because the
loop \(\gamma\) is a vanishing cycle for the conic
fibration. Therefore \(\pi_1(W)\) is a quotient of
\(\pi_1(U)=\ZZ^3\), hence abelian. \qedhere

\end{proof}
\begin{lem}\label{lma:negvtx2}
Let \(\phi\colon \partial(D^2\times S^1)\to S^1\vee S^1\) be the
attaching map for the solid torus to the 1-skeleton. Then, after
possibly precomposing with a diffeomorphism of the solid torus and
postcomposing with a conjugation,
\(\phi_*\colon\ZZ^2\to\ZZ\star\ZZ=\langle a,b\rangle\) satisfies
\(\phi_*(1,0)=aba^{-1}b^{-1}\), \(\phi_*(0,1)=1\). This determines
\(\phi\) completely up to free homotopy (and precomposition by a
diffeomorphism of the solid torus).
\end{lem}
\begin{proof}
The fact that \(\phi\) is determined by \(\phi_*\) follows from the
fact that \(S^1\vee S^1\) is an Eilenberg-MacLane space.

The image of \(\phi_*\) is a subgroup of a free group and hence
free, however it is also abelian, so it is either trivial or has
rank 1. In other words, \(\phi_*(1,0)=c^m\) and \(\phi_*(0,1)=c^n\)
for some \(c\in\langle a,b\rangle\) and \(m,n\in\ZZ\). Suppose that
\((1,0)\) is the loop in \(T^2\) which bounds a disc in the solid
torus.

Van Kampen's theorem tells us that
\(\pi_1(W)=\langle a,b\ |\ c^n\rangle\). By Lemma \ref{lma:zariski},
\(\pi_1(W)\) is abelian. Therefore \(aba^{-1}b^{-1}\) is contained in
the normal subgroup generated by \(c^n\), so
\(aba^{-1}b^{-1}=hc^{nn'}h^{-1}\) for some \(h\in\langle a,b\rangle\),
\(n'\in\ZZ\). However, \(aba^{-1}b^{-1}\) is not a nontrivial power in
the free group, so \(n=n'=1\) and \(aba^{-1}b^{-1}=hch^{-1}\).

The Dehn twist around the loop \((0,1)\) in \(T^2\) extends to a
diffeomorphism of the solid torus, and precomposing with a power of
this diffeomorphism allow us to change \(m\) by a multiple of
\(n\). Since \(n=1\), we can achieve \(m=0\). This proves the
lemma. \qedhere

\end{proof}

\section{Dual complexes}
\label{sct:dualcomplex}
\subsection{Definition}

Let \(Y\) be a pure-dimensional simple normal crossing variety of
dimension \(n-1\) (see Definition 1.8 of \cite{Kollar2013}), for
example, a simple normal crossing divisor as in Section
\ref{sct:ncd}. Note that \(Y\) is stratified by the intersections of
its irreducible components \(Y_i\), i.e.
\[S_d(Y)=\bigcup_{J\subseteq I\ :\ |J|=n-d}Y^0_J,\] where
\(Y^0_J:=\left(\bigcap_{j\in
    J}Y_j\right)\setminus\left(\bigcup_{j\not\in J}Y_j\right)\).

\begin{Definition}[Dual complex]\label{dfn:dualcomplex}
The {\em dual complex} \(\mathcal{D}(Y)\) of \(Y\) is a regular
\(\Delta\)-complex {\cite[Section 2.1]{H02}} whose vertices are in
correspondence with the irreducible components \(Y_i\) and whose
\(d\)-cells correspond to connected components of
\(S_{n-d-1}(Y)\). Like any \(\Delta\)-complex, the dual complex is
stratified (the \(d\)-stratum is the union of its open \(d\)-cells).

\end{Definition}
\begin{Definition}[Maximal intersection]\label{dfn:maximalintersection}
We say that \(Y\) has {\em maximal intersection} if it admits a
stratum of dimension zero. Equivalently, the corresponding cells of
\(\mathcal{D}(Y)\) have real dimension \(\dim_{\CC}(Y)\), and we say
that \(\mathcal{D}(Y)\) has maximal dimensional.

\end{Definition}
\subsection{Dual boundary complexes of affine varieties}

Let \(Y\subset X\) be a simple normal crossing divisor of dimension
\(n-1\). In this section, we study the homotopy type of the dual
complex \(\mathcal{D}(Y)\), under the assumption that \(X \setminus
Y\) is affine.

First, recall the following result due to Danilov {\cite[Proposition
3]{Danilov1975}}.

\begin{Proposition}\label{prp:danilov}
If one of the irreducible components of \(Y\) is ample, then
\(\mathcal{D}(Y)\) has the homotopy type of a bouquet of spheres.

\end{Proposition}
The hypothesis of Proposition \ref{prp:danilov} is too restrictive for
our purposes (for instance, it does not include all toric
boundaries). Moreover, the statement does not provide any control on
the number of spheres in the bouquet.

The following propositions can be regarded as generalizations and
refinements of Proposition \ref{prp:danilov}, and they are inspired by
\cite{Payne2013} and \cite{KollarXu2016}.

\begin{Proposition}[Rational cohomology]\label{prp:rationalcohomology}
Let \(Y \subset X\) be a simple normal crossing divisor of dimension
\(n-1\) and \(W := X \setminus Y\) be an affine variety. Then,
\begin{equation} \label{vanishingPayne} h^i(\mathcal{D}(Y), \QQ)=0
\qquad 0 < i < n-1. \end{equation} If \(X\) has Hodge coniveau
\(\geq 1\), i.e. \(h^{0,i}(X)=0\) for all \(i >0\), and \(Y_J: =
\bigcap_{j \in J} Y_j\) does not admit global holomorphic canonical
sections for any \(J \subseteq I\) (e.g. if \(X\) and \(Y_J\) are
rationally connected), then
\begin{equation} \label{topgroup}
h^{n-1}(\mathcal{D}(Y), \QQ)=h^0(X, K_X + Y).
\end{equation}
\end{Proposition}
\begin{proof}
  The vanishing (Equation \eqref{vanishingPayne}) of the rational
  cohomology is noted in {\cite[Section 6]{Payne2013}}.

  In order to prove Equation \eqref{topgroup}, we identify the
  \((n-1)\)th cohomology group of \(\mathcal{D}(Y)\) with the
  \((n-1)\)th cohomology group of the structure sheaf
  \(\mathcal{O}_Y\). Note indeed that the cohomology of
  \(\mathcal{O}_Y\) is computed by a spectral sequence whose page
  \(E_1\) is given by
  \[E^{p,q}_1:= \bigoplus_{J \subseteq I, |J|=q+1} H^{p}(Y_J,
    \mathcal{O}_{Y_J}), \] and which degenerates at the page \(E_2\);
  see {\cite[Proof of Proposition 1.5.3]{Friedman1983}}. Since
  \(H^{\dim Y_J}(Y_J, \mathcal{O}_{Y_J})=0\) for any
  \(J \subseteq I\), and \(\dim Y_J>0\), we have that
  \[ H^{n-1}(Y, \mathcal{O}_Y)=E^{0,n-1}_2 = H^{n-1}(\mathcal{D}(Y),
    \QQ), \] since the complex \((E^{0, *}_1, d_1)\) computes the
  cellular cohomology of \(\mathcal{D}(Y)\).

Further, the short exact sequence \[ 0 \to \mathcal{O}_X(-Y) \to
\mathcal{O}_X \to \mathcal{O}_Y \to 0\] induces the following
isomorphism in cohomology \[H^{n-1}(Y, \mathcal{O}_Y) \simeq
H^{n}(X, \mathcal{O}_X(-Y))\simeq H^0(X, K_X + Y)^{\vee},\] since by
hypothesis \(H^i(X, \mathcal{O}_X)=0\) for \(i >0\). We conclude
that \(H^{n-1}(\mathcal{D}(Y), \QQ)\simeq H^0(X, K_X + Y)^{\vee}\), as
desired. \qedhere

\end{proof}
\begin{Proposition}[Fundamental group]\label{prp:pi1}
Let \(Y \subset X\) be a simple normal crossing divisor of dimension
\(\geq 2\) and \(W := X \setminus Y\) be an affine variety. If \(X\)
is simply connected (e.g. if \(X\) is rationally connected), then
\(\mathcal{D}(Y)\) is so as well.
\end{Proposition}
\begin{proof}
  By the Lefschetz hyperplane theorem, \(\pi_1(X) \simeq
  \pi_1(Y)\). Note also that there is a natural surjective map
  \(\pi_1(Y) \twoheadrightarrow \pi_1(\mathcal{D}(Y))\), induced for
  instance by the evaluation map in Definition \ref{dfn:evaluation};
  see also {\cite[Lemma 26]{KollarXu2016}}. We conclude that
  \(\pi_1(\mathcal{D}(Y))\) is a quotient of \(\pi_1(X)\). \qedhere

\end{proof}
\begin{Proposition}[Homotopy type]\label{prp:htpytype}
  Let \(Y = \bigcup_{i \in I} Y_i \subset X\) be a simple normal
  crossing divisor such that \(X \setminus Y\) is an affine
  variety. Suppose that the full set of hypotheses of Proposition
  \ref{prp:rationalcohomology} hold.
\begin{itemize}
\item If \(\dim(X)=2\), then \(\mathcal{D}(Y)\) is a graph with \(h^0(X,
K+X+Y)\) loops.
\item If \(\dim(X)=3\) and \(X\) is simply-connected, then
\(\mathcal{D}(Y)\) has the homotopy type of a bouquet of \(h^0(X,
K_X + Y)\) spheres of dimension \(2\), unless it is contractible.
\end{itemize}
\end{Proposition}
\begin{proof}
  In the two-dimensional case, the statement follows immediately from
  Proposition \ref{prp:rationalcohomology}.(\ref{topgroup}). Suppose
  now that \(\dim(X)=3\). Since \(X\) is simply-connected, then
  \(\mathcal{D}(Y)\) is simply-connected by Proposition
  \ref{prp:pi1}. For dimensional reasons, \(\mathcal{D}(Y)\) has
  torsion-free integral homology. Therefore Proposition
  \ref{prp:rationalcohomology} says that \(\mathcal{D}(Y)\) has the
  integral homology of a bouquet of \(h^0(X, K_X + Y)\) spheres of
  dimension \(2\), or of a point. In the former case, Hurewicz's
  theorem then gives a continuous map from such a bouquet to
  \(\mathcal{D}(Y)\) inducing an isomorphism on homology. Since
  \(\mathcal{D}(Y)\) is simply-connected, Whitehead's theorem implies
  that this map is a homotopy equivalence. \qedhere

\end{proof}
\section{The evaluation map}
\label{sct:evaluationmap}

In this section, we construct a map (the {\em evaluation map}) from
the link of a simple normal crossing divisor to the dual complex and
show that, with suitable choices, it is a coisotropic fibration.

We use the following standard notation: if \(S,T\)
are sets then \(S^T\) denotes the space of maps \(T\to S\).

\subsection{Evaluation map}

Let \(Y=\bigcup_{i\in I}Y_i\subset X\) be a simple normal crossing
divisor. Here and in the following, we can assume that
\(Y_J=\bigcap_{j \in J}Y_j\) is connected for any \(J \subset I\) (the
assumption is not essential but it makes the notation lighter and
allows us to realise \(\mathcal{D}(Y)\) as the image of the evaluation
map defined below. In fact, it can always be achieved via a sequence
of blowups along connected components of the strata of \(Y\)).

\begin{Definition}[Full partition of
  unity]\label{independentpartition}
  Let \(\{\chi_i\}_{i\in I}\) be a partition of unity subordinate to
  the open cover \(\{U_i\}_{i\in I}\). Then \(\{\chi_i\}_{i\in I}\) is
  called {\em full} if the following property holds: for any
  \(J\subseteq I\) such that
  \(U_J\coloneqq\bigcap_{j\in J}U_j\neq \emptyset\), the map
  \(\chi^J\colon U_J\to [0,1]^J\), given by \(\chi^J(j)=\chi_j(x)\),
  is surjective onto the open standard simplex in \((0,1]^J\) given by
  the equation \(\sum_{j\in J}y_j=1\).

  Concretely, this means that the values \(\chi_j(x)\) with
  \(j\in K\subseteq J\) do not impose any constraints on the values
  that the other functions \(\chi_{j'}(x)\) with
  \(j'\in J\setminus K\) can attain, with the exception of the
  relation \(\sum_{j\in J}\chi_j(x)=1\).

\end{Definition}
\begin{Definition}[Plumbing neighbourhood]\label{dfn:plumbingnbhd}
A {\em plumbing neighbourhood} \(N\) of \(Y\) is the union of
tubular neighbourhoods \(N_i\) of \(Y_i\) for all \(i \in I\).

\end{Definition}
\begin{Definition}[Evaluation map]\label{dfn:evaluation}
  Let \(\{\chi_i\}_{i\in I}\) be a full partition of unity subordinate
  to the open cover \(\{N_i\}_{i\in I}\) of the plumbing neighbourhood
  \(N\). Then the {\em evaluation map}
  \[ev: \Link(X \setminus Y) \subset N \to \RR^I \] is given by
  \(ev(x)(i)=\chi_i(x)\).

\end{Definition}
\begin{lem}
The image of \(ev\) is homeomorphic to the dual complex
\(\mathcal{D}(Y)\).
\end{lem}
\begin{proof}
  Let \(N_J := \cap_{i\in J}N_i\) and
  \(N_J^\circ := N_J \setminus \bigcup_{i \notin J} N_i\). By
  definition of fullness, the image of \(N_J^{\circ} \cap \Link(Y)\)
  via \(ev\) is the convex hull of $e_j$ with \(j \in J\), where
  \(\{e_i\}_{i \in I}\) is the standard basis of \(\RR^{|I|}\). In
  particular, it is a standard simplex of dimension \(|J|-1\), and it
  corresponds to the \((|J|-1)\)-cell of \(\mathcal{D}(Y)\) associated
  to $Y_J$. \qedhere

\end{proof}
We will see that if \((X,\omega)\) is a symplectic manifold then the
partition of unity can be chosen to make \(ev\) into a generically
Lagrangian coisotropic fibration.

\begin{Remark}[Relation with other work]\label{rmk:otherwork}
The evaluation map \(ev\) appears in the geometric P=W conjecture
{\cite[Conjecture 1.1]{KatzarkovNollPanditEtAl2015}}. Further, a
similar map, called \(\operatorname{Log}_{\mathcal{V}}\), is used to
define the topology of hybrid spaces in {\cite[Definition
2.3]{BoucksomJonsson2017}}. The key point that we address in this
paper is that \(ev\) can be {\em adapted} to a symplectic form
\(\omega\), meaning that it can be turned into a generically
Lagrangian fibration, unique up to homotopy.

\end{Remark}

\subsection{Evaluation map as a coisotropic fibration}

We now show that the evaluation map can be made into a coisotropic
fibration on the contact hypersurface \(\Link(X\setminus Y)\).

Let \(\nu_i\) be the normal bundle of \(Y_i\) in \(X\) and pick an
\(\omega\)-regularisation
\(\{\Psi_i\colon \OP{nbhd}_{\nu_i}(Y_i)\to X\}_{i\in I}\) of \(Y\) as
in Section \ref{sct:sympnbhd}; write
\(N_i=\OP{Image}(\Psi_i)\). Recall that \(\Link(X\setminus Y)\) can be
written as \(f^{-1}(\epsilon)\) for the \(Y\)-compatible function
\(f(x)=\sum_{i\in I}\log(\GG(\mu_i(x)))\) (where \(\GG\) is defined in
Equation \eqref{eq:cutoff2}, \(\mu_i\) is the moment map for the
circle action which rotates the fibres of the normal bundle to
\(Y_i\), and \(\epsilon>0\) is a small parameter).

\begin{Definition}\label{def:cutoff}
Let \(\FF\colon\RR\to[0,1]\) be a smooth cutoff function satisfying
\begin{equation}\label{eq:cutoff} \FF(x)=\begin{cases} 1 & \text{ if
}x \leq 0\\ 0 & \text{ if }x \geq \epsilon \end{cases}
\end{equation} and which is positive and strictly decreasing on
\([0,\epsilon)\).

\begin{center}
\begin{tikzpicture}[baseline=0cm]
\draw (2.3,0) -- (0,0) -- (0,2.3);
\node at (2.3,0) [right] {\(x\)};
\node at (0,2.3) [left] {\(\FF(x)\)};
\node at (2,0) [below] {\(\epsilon\)};
\draw (0,2) to[out=0,in=180] (2,0);
\end{tikzpicture}
\end{center}
\end{Definition}

\begin{Proposition}\label{prp:evalmap}
Assume that $Y$ has maximal intersection. Then, there exists a
partition of unity \(\{\chi_i\}_{i \in I}\) subordinate to the open
cover \(\{N_i\}_{i \in I}\), such that the evalution map \(ev:
\Link(X \setminus Y) \to \mathcal{D}(Y)\) is a generically
Lagrangian coisotropic fibration (compatible with the stratification
of \(\mathcal{D}(Y)\)).
\end{Proposition}
\begin{proof}
  Let \(\FF_i:=\FF\circ\mu_i\) and choose the following partition of
  unity on \(N\) subordinate to \(\{N_i\}_{i\in I}\):
  \[ \chi_i(x) := \frac{\FF_i(x)}{\sum_{j \in I}\FF_j(x)}.\] Because
  the functions \(\mu_i\) generate a Hamiltonian torus action, their
  Poisson brackets vanish. Since the functions \(\chi_i\) depend only
  on the \(\mu_j\), they also Poisson commute with one another, and
  with \(\prod_{i\in I}\GG\circ\mu_i\), so the restrictions of these
  functions to the link define a coisotropic fibration over the dual
  complex. 
  
  Note further that \(\chi_i \equiv 0\) on 
  \(N^{\circ}_J:=N_J\setminus\bigcup_{k\not\in J}N_k\) if and only if \(i \notin J\). Hence, the evaluation map projects
  \(N^{\circ}_J\cap\Link(X\setminus Y)\) submersively onto the open
  \((|J|-1)\)-cell in \(\mathcal{D}(Y)\) corresponding to \(Y_J\). The
  fibre \(ev^{-1}(p)\) over a point \(p\) of this cell is coisotropic
  of codimension \(|J|-1\). The isotropic leaves of \(ev^{-1}(p)\) are
  precisely the fibres of the projection
  \(\pi_J\colon N^{\circ}_J\cap\Link(X\setminus Y)\to Y_J\), which are
  \(|J|\)-dimensional tori. In particular, \(ev\) is generically
  Lagrangian if \(Y\) has maximal intersection. \qedhere
\end{proof}

\begin{Remark}\label{rmk:rigidity}
  Note that the functions \(\{\mu_j\}_{j\in J}\) generate a
  Hamiltonian \(n\)-torus action on the regions \(N^{\circ}_J\) with
  \(|J|=n\). The Reeb flow on \(N^{\circ}_J\cap\Link(X\setminus Y)\)
  is an \(\RR\)-action inside that torus, and for a dense set of
  points in the corresponding \(n\)-cell of the dual complex, this
  \(\RR\)-action has dense image in the \(n\)-torus. In particular,
  Lemma \ref{lma:rigidity} applies to the evaluation map over these
  regions, and hence to any Lagrangian torus fibration which is a
  refinement of it.
\end{Remark}

\section{Analogue of non-archimedean SYZ fibration}
\label{sct:nonarch}

In Theorem \ref{thm:nonarchLagrafibra} we construct a smooth
Lagrangian torus fibration that can be regarded as a symplectic
analogue (or dual) of the non-archimedean SYZ fibration constructed in
{\cite[Theorem 6.1]{NicaiseXuYu}}. The base of this Lagrangian
fibration has the homotopy type of the smooth (or affinoid) locus of
the non-archimedean SYZ fibration. In particular, the integral affine
structures induced on their bases have \emph{dual} monodromy in the
sense of Proposition \ref{prop:monodromy}.

We first give a biased introduction to Berkovich spaces and recall the
notion of non-archimedean SYZ fibration. For further detail we refer
the interested reader to \cite{Berkovich1999}, \cite{T07},
\cite{Nicaise2016} and \cite{NicaiseXuYu}.

\subsection{A brief review of Berkovich spaces}

Let $X$ be a smooth connected variety over \(\CC\).

We denote by $X^{an}$ the {\em Berkovich analytification} of $X$. As a
set, \(X^{an}\) is the space of rank-one semi-valuations of the
fraction field $\CC(X)$ of $X$ that extend the trivial valuation of
$\CC$, i.e. the set of functions $v: \CC(X) \to \overline{\mathbb{R}}$
with the properties that:
\begin{enumerate}
\item $v(f \cdot g)=v(f) + v(g)$ for all $f, g \in \CC(X)$;
\item $v(f+g)\leq \min\{v(f), v(g)\}$ for all $f, g \in \CC(X)$;
\item $v(h)=0$ for all $h \in \CC^*$.
\end{enumerate} 

\begin{example}
  Let \(D\) be a divisor on a birational modification of \(X\). There
  is a semi-valuation which associates to any rational function on
  \(X\) its order of vanishing along \(D\). These {\em divisorial
    valuations} form an important class of semi-valuations.
\end{example}


The Berkovich analytification is endowed with the coarsest topology
for which the following maps are continuous:
 \begin{itemize}
 \item the {\em analytification morphism} $i: X^{an} \to X$, which
   sends any semi-valuation $v$ to its kernel, i.e. the schematic
   point defined by the ideal of local regular functions $g$ such that
   $v(g)=+ \infty$. Note that here $X$ is endowed with the Zariski
   topology.
 \item the {\em norm maps} $\|f\|: X^{an} \to \overline{\mathbb{R}}$,
   given by $\|f\|(v)=\exp(-v(f))$, for any $f \in \CC(X)$.
 \end{itemize} 

 In order to illustrate some aspects of the topology of Berkovich
 spaces, we describe in detail a fundamental example: the
 analytification of an algebraic torus \((\CC^*)^n\).

 \begin{example}[Algebraic torus]\label{Ex: algebraic torus}
   Let
   \[T \coloneqq(\CC^*)^n \simeq \CC^* \otimes N \simeq
     \Spec(\CC[M])\] be an algebraic torus with character lattice $M$
   and cocharacter lattice $N = M ^{\vee}$ (i.e. the lattice of
   one-parameters
   subgroups). 


   Define the {\em tropicalization} as the map 
   \begin{align*}
     \rho_T : \, T^{an} \to N_{\mathbb{R}} \coloneqq N\otimes \mathbb{R} \quad \quad v \mapsto (M \to \mathbb{R}: m \mapsto v(m)).
   \end{align*}
   It should be regarded as a non-archimedean analogue of the moment
   map for the standard Hamiltonian action of the real torus
   $(S^1)^n \subset T$
   \[ 
     \mu_T: T \to M_{\mathbb{R}} \qquad x=(x_1, \ldots, x_n) \mapsto \bigg(\frac{1}{2}|x_1|^2, \ldots, \frac{1}{2}|x_n|^2 \bigg).
   \]
   Notice first that among the semi-valuations in $T^{an}$ there are
   those which associate to any rational function $f$ the value
   $-\log|f(x)|$ for $x \in T$. Hence, there is a copy of $T$ which
   sits inside $T^{an}$; however, it is equipped with the discrete
   topology. The restriction of the tropicalization to $T$ yields the
   following map
   \[
     T \hookrightarrow T^{an} \xrightarrow{\rho_T} N_{\mathbb{R}}\qquad x=(x_1, \ldots, x_n) \mapsto (-\log |x_1|, \ldots, -\log |x_n|),
   \]
   which coincides with the moment map $\mu_T$ up to a diffeomorphism
   of the image space.
   
   The fibres of $\rho_T$ should be viewed as non-archimedean
   Lagrangian tori, although at the moment it is not clear what a
   rigorous notion of non-archimedean Lagrangian analytic space should
   be in general.
   
   A common feature of the maps $\rho_T$ and $\mu_T$ is that they both
   induce an integral affine structure on their bases. The affine
   structure induced by $\rho_T$ is constructed as follows: the norm
   maps $\|m\|$ of the invertible functions $m \in M$ are constant
   along the fibre of $\rho_T$, and so they descend to
   $N_{\mathbb{R}}$, and define a sheaf of affine functions; see also
   \cite[Section 4.1, Lemma 1]{KontsevichSoibelman2006}. Via the
   identification of the tangent spaces of the points of
   $N_{\mathbb{R}}$ with $N_{\mathbb{R}}$ itself, the flat integer
   lattice of this affine structure can be identified with the
   cocharacter lattice $N$; see Definition \ref{dfn:Zaff}.

   Classically, the base of the moment map $\mu_T$ carries an integral
   affine structure given by the dual of the lattice of the 1-forms
   $\int_{\gamma} \omega$, where $\omega$ is the standard symplectic
   form on $T$, and $\gamma$ is a loop of the fibre $\mu^{-1}(m)$,
   with $m \in M_{\mathbb{R}}$. According to the identification of the
   tangent spaces of the points of $M_{\mathbb{R}}$ with
   $M_{\mathbb{R}}$, the flat integer lattice of this affine structure
   can be identified with the character lattice $M$.

   The natural pairing
   $\langle \cdot , \cdot \rangle : N \times M \to \mathbb{Z}$ shows
   that the two affine structures are dual to each other.
\end{example}

\subsection{The non-archimedean SYZ fibration}

Just as the moment map $\mu_T$ is the local model for a Lagrangian
submersion by the Arnold-Liouville theorem, the tropicalization should
be taken as the local model of a non-archimedean Lagrangian
submersion, also known as affinoid torus fibration.

\begin{Definition}[Affinoid torus fibration] \cite[Paragraph
  3.3]{NicaiseXuYu} A morphism $f: \mathfrak{X} \to B$ from an
  analytic space $\mathfrak{X}$ to a topological space $B$ which is
  locally modelled on $\rho_T$ is called {\em affinoid torus
    fibration}, i.e. $B$ can be covered by open subsets $U$ such that
  there exists an open $V \subseteq N_{\mathbb{R}}$ and a commutative
  diagram

\begin{center}
\begin{tikzpicture}
\node (A) at (0,2) {\(f^{-1}(U)\)};
\node (B) at (3,2) {\(\rho_T^{-1}(V)\)};
\node (C) at (0,0) {\(U\)};
\node (D) at (3,0) {\(V\)};
\draw[->,thick] (A) -- (C) node [midway,left] {\(f\)};
\draw[->,thick] (B) -- (D) node [midway,right] {\(\rho_T\)};
\draw[->,thick] (A) -- (B);
\draw[->,thick] (C) -- (D);
\end{tikzpicture}
\end{center}

such that the upper horizontal map is an isomorphism of analytic
spaces and the lower horizontal map is a homeomorphism.
\end{Definition}


Example \ref{Ex: algebraic torus} is the local model for an affinoid torus fibration, and many of the local constructions described there can
be globalised. A proof of these facts goes beyond the purposes of this
paper, so we will give only statements and refer the interested reader
to the literature.

Let \(Y=\bigcup_{i\in I}Y_i\) be a
simple normal crossing divisor of a  smooth complex projective variety $X$. 

\begin{Definition} The {\em centre} of a semi-valuation $v$ is the
  schematic point defined by the ideal of local regular functions $g$
  such that $v(g)>0$.  The {\em analytic generic fibre} of the formal
  completion of \(X\) along \(Y\), denoted by \(\mathfrak{X}_{\eta}\),
  is the subset of \(X^{an}\) of valuations that admit center on $Y$
  (see \cite[Paragraph 1.1.11]{T07}).
\end{Definition}

\begin{thm}{\em\cite[Theorem 3.26 and Corollary 3.27]{T07}}\label{thm: strong deformation}
  There exist strong deformation retracts
\begin{align*}
  \rho_{X, Y}:  X^{an} & \to \Cone(\mathcal{D}(Y))\\
  \rho_{\mathfrak{X}_{\eta}} \coloneqq \rho_{X, Y}|_{\mathfrak{X}_{\eta}} :  \mathfrak{X}_{\eta} & \to \Cone(\mathcal{D}(Y)) \setminus \{\mathrm{vertex}\}.
\end{align*}
\end{thm}

\begin{Assumption}\label{assumption:logCalabi-Yau1-dim}
  Suppose that \(Y=\bigcup_{i\in I}Y_i\) is a simple normal crossing
  divisor with maximal intersection and log Calabi-Yau along
  1-dimensional components \(Y_K\) (\(K \subset I\), \(|K|=n-1\)),
  i.e. each \(Y_K\) is a copy of \(\cp{1}\) which hits precisely two
  points \(Y_L\) (\(L \subset I\), \(|L|=n\)) in the 0-stratum of
  \(Y\).
 \end{Assumption}
 \begin{thm}{\em \cite[Theorem 6.1]{NicaiseXuYu}}\footnote{ In non-archimedean
     mirror symmetry it is customary to normalise the points in $X^{an}$ by the
     valuation of a local equation of $Y$. For this reason, in \cite{NicaiseXuYu}
     the base of $\rho_{X, Y}$ is the dual complex instead of its cone. However,
     we prefer the non-normalized version because it makes the relation with the
     Lagrangian torus fibration \(\Phi\) defined in Section \ref{sec:1in8} neater.}
   Suppose that Assumption \ref{assumption:logCalabi-Yau1-dim} holds
   and that $Z$ is the union of the strata of codimension \(\geq 2\)
   in \(\Cone (\mathcal{D}(Y))\). Then the retract
   $\rho_{\mathfrak{X}_{\eta}}$ is an affinoid torus fibration over
   $\Cone(\mathcal{D}(Y)) \setminus Z$.
 \end{thm}

 \begin{Remark}\label{rmk:non-archimedeanSYZ} If $Y$ is a singular
   fibre with maximal intersection of a minimal semistable
   degeneration of Calabi-Yau varieties, then $\rho_{X, Y}$ is the
   so-called {\em non-archimedean SYZ fibration}. The reason for this
   name is that conjecturally it should be possible to dualise these
   fibration on the affinoid locus, and construct a non-archimedean
   mirror. A GAGA principle for Berkovich spaces would give back the
   classical mirror, bypassing the difficult task of constructing a
   geometric special Lagrangian fibration in the first place. See
   \cite{KontsevichSoibelman2006}.
\end{Remark}

\begin{example}[Toric variety] \cite[Example
  3.5]{NicaiseXuYu} \label{Ex: toric variety} Let $X$ be a smooth
  toric variety with fan $\Sigma \subset N_{\mathbb{R}}$, and $Y$ its
  toric boundary. One can show that
  \begin{align*}
    \rho_{\mathfrak{X}_{\eta}}  :  \mathfrak{X}_{\eta} & \to \Cone(\mathcal{D}(Y)) \simeq \Sigma \subseteq N_{\mathbb{R}},\\
    v & \mapsto (M \to \mathbb{R}: m \mapsto v(m)).
  \end{align*}
  The base of $\rho_{\mathfrak{X}_{\eta}}$ is endowed with an integral
  affine structure, which can be identified with the lattice $N$ as in
  Example \ref{Ex: algebraic torus} (in this case the integral affine
  structure actually extends through the codimension-two locus $Z$).
\end{example}

\subsection{Non-archimedean monodromy}\label{sec:non-archmonodromy}

Let \(X\) be a smooth complex projective variety of complex dimension
\(n\), and \(Y=\bigcup_{i\in I}Y_i\subset X\) be a simple normal
crossing divisor satisfying Assumption
\ref{assumption:logCalabi-Yau1-dim}.

The charts of the affinoid torus fibration
$\rho_{\mathfrak{X}_{\eta}}$ in Theorem \ref{thm: strong deformation}
are endowed with the integral affine structures defined in Example
\ref{Ex: algebraic torus}, which glue to a global affine structure on
the whole $\Cone(\mathcal{D}(Y)) \setminus Z$; see \cite[Section
6]{NicaiseXuYu} for a detailed proof of the gluing. In this section,
we show that the monodromy of this integral affine structure can be
defined in purely topological terms.

Following \cite{NicaiseXuYu}, we present an atlas of $\Cone(\mathcal{D}(Y)) \setminus Z$, made of charts of the affinoid torus fibration of two different types.
\begin{enumerate}
\item An open set of first type  consists of a top dimensional open
cone \(\Cone(\sigma_L)\). The $(n-1)$-cell $\sigma_L \subseteq \mathcal{D}(Y)$ corresponds to the 0-dimensional component $Y_L$, for some $L \subseteq I$ with $|L|=n$. In particular, the pair $(X, Y)$ is locally modelled at $Y_L$ on the affine space $\CC^n$ with its coordinate hyperplanes $\Delta$. 

By construction (cf \cite[Section 2.4]{NicaiseXuYu}), the fibration $\rho_{\mathfrak{X}_{\eta}}$ over \(\Cone(\sigma_L)\) coincides with $\rho_{\CC^n, \Delta}$, described in Example \ref{Ex: toric variety}. In particular, the integral affine structure of the basis can be identified with the lattice of one-parameter subgroups of $\CC^n$, and we denote it by $N(L)$.
\item An open set of second type is the open star
\(\operatorname{Star}_K\) of the codimension-one open cells
\(\Cone(\sigma_K)\). The $(n-2)$-cell $\sigma_K \subseteq \mathcal{D}(Y)$ corresponds to the 1-dimensional component $Y_K$, for some $K \subseteq I$ with $|K|=n-1$. 

By \cite[Proposition 5.4]{NicaiseXuYu}, Assumption \ref{assumption:logCalabi-Yau1-dim} implies that there exists a (formal) toric tubular neighbourhood of $Y_K$. More precisely, the formal completion of \(X\) along \(Y_K\) is isomorphic to the formal
completion of a toric vector bundle \(\nu_K\) of
rank \(n-1\) on \(Y_K \simeq \cp{1}\) along its zero section.
 
Again, the fibration $\rho_{\mathfrak{X}_{\eta}}$ over \(\operatorname{Star}_K\) coincides with $\rho_{\nu_K, \Delta}$, where $\Delta$ is the toric boundary of $\nu_K$. In particular, the integral affine structure on the basis can be identified with the cocharacter lattice of \(\nu_K\), and we denote it by $N(K)$.
\end{enumerate}

Note that the inclusion $Y_L \subset Y_K$ induces the embedding $\Cone(\sigma_L) \subset \operatorname{Star}_K$, and an identification of the respective integral affine structures given by the linear map $\beta_{LK}: N(L) \to N(K)$.

Without loss of generality assume that $\Cone(\mathcal{D}(Y))
\setminus Z$ is connected, and fix $b_0$ a base point in $\operatorname{Star}_{K}$. For any loop $\gamma$  in $\Cone(\mathcal{D}(Y))
\setminus Z$ based at $b_0$, consider now an ordered finite sequence of open sets covering $\gamma$ 
\begin{center}
  \begin{tikzpicture}
  \node at (0.25,1) {\(\operatorname{Star}_{K} = \operatorname{Star}_{K_0}\)};
  \node at (1,0.5) {\rotatebox[origin=c]{135}{$\subset$}};
  \node at (1.75,0) {\(\Cone(\sigma_{L_1})\)};
  \node at (2.5,0.5) {\rotatebox[origin=c]{45}{$\subset$}};
  \node at (3.2,1) {\(\operatorname{Star}_{K_2}\)};
  \node at (3.75,0.5) {\rotatebox[origin=c]{135}{$\subset$}};
  \node at (4.25,0.5) {\(\ldots\)};
  \node at (4.75,0.5) {\rotatebox[origin=c]{45}{$\subset$}};
  \node at (5.65,1) {\(\operatorname{Star}_{K_{n-2}}\)};
  \node at (6,0.5) {\rotatebox[origin=c]{135}{$\subset$}};
  \node at (6.8,0){ \(\Cone(\sigma_{L_{n-1}}\))};
  \node at (7.5,0.5) {\rotatebox[origin=c]{45}{$\subset$}};
  \node at (8.5,1) {\(\operatorname{Star}_{K_n}= \operatorname{Star}_{K}\)};
  \end{tikzpicture}
  \end{center}

\begin{Definition}[Non-archimedean monodromy]\label{Non-archimedean monodromy}
The {\em monodromy representation} of the integral structure of the affinoid fibration $\rho_{\mathfrak{X}_\eta}$ is the group homomorphism 
\[\rho_{\text{non-arch}}:
\pi_1(\Cone(\mathcal{D}(Y))\setminus Z) \to \operatorname{GL}(N(K)),\]
given by \(\rho_{\text{non-arch}}([\gamma])= \beta_{L_{n-1}K}\circ \beta^{-1}_{L_{n-1}K_{n-2}} \ldots \circ \beta_{L_1 K_2} \circ \beta^{-1}_{L_1 K}\).
\end{Definition}

\begin{Remark}
In \cite[Proposition 5.4]{NicaiseXuYu}, the algebraic tubular
neighbourhood theorem for one-dimensional components $Y_K$ is proved under the assumption that the conormal
bundle of \(Y_K\) is ample. This condition can be always achieved by a
sequence of stratum blow-ups at \(Y_L \subset Y_K\) (or in the strict transform of \(Y_K\)) \[ X^m
\xrightarrow{\pi^m} X^{m-1}\to \ldots \to X^0=X. \] Note that stratum blow-ups do not alter the integral affine structure on
\(\Cone(\sigma_L)\). The positivity assumption can be actually removed
in the following way. The case of ample conormal bundle implies
that the formal completion \(\widehat{X}^m_{Y^m_K}\) of \(X^m\) along
the strict transform \(Y^m_K\) of \(Y_K\) is isomorphic to the formal
completion \(\widehat{\nu}^m_{K}\) of the toric vector bundle
\(\nu^m_{K}\) along its zero section. The isomorphism,
\(f^m\colon\widehat{\nu}^m_K\to \widehat{X}^m_{Y^m_K}\), is
constructed in \cite[Proposition 5.4]{NicaiseXuYu}. In particular, the
stratum contracted by \(\pi^m\) is the image under \(f^m\) of a
torus-invariant divisor in \(\widehat{\nu}^m_K\). The corresponding
toric blow-down is again the formal completion of a toric vector
bundle \(\nu^{m-1}_K\). Now, the morphism \(\pi^m \circ f^m\) factors
through a morphism \(f^{m-1}: \widehat{\nu}^{m-1}_K \to
\widehat{X}^{m-1}_{Y^{m-1}_K}\) by the universal property of blow-ups,
and we observe that \(f^{m-1}\) is an isomorphism, as required.
\end{Remark}

\subsection{Topological monodromy}\label{sec:1in8}

Let \((X, \omega)\) be a smooth complex projective variety of complex
dimension \(n\), and \(Y=\bigcup_{i\in I}Y_i\subset X\) a simple
normal crossing satisfying Assumption
\ref{assumption:logCalabi-Yau1-dim}. We are particularly interested in
the symplectic tubular neighbourhood of the complex 1-dimensional
stratum: our goal is to construct a Lagrangian torus fibration over
this neighbourhood, and compare its monodromy with the non-archimedean
monodromy (Definition \ref{Non-archimedean monodromy}).

To this end, first make a perturbation of the
symplectic form so that \(Y\) admits an \(\omega\)-regularisation by
\cite{MTZ}. For each \(K\subset I\) of size \(n-1\) and any \(L \subset I\) of size \(n\) containing \(K\) with \(Y_K\neq\emptyset\),
a \(\omega\)-regularisation gives us commutative squares
  \begin{center}
  \begin{tikzpicture}
  \node (B) at (3,2) {\(N_K\)};
  \node (A) at (0,2) {\(\OP{nbhd}_{\nu_K}(Y_K)\)};
  \node (D) at (3,0) {\(N_L\)};
  \node (C) at (0,0) {\(\OP{nbhd}_{\nu_L}(Y_L)\)};
  \draw[->,thick] (C) -- (A);
  \draw[->,thick] (D) -- (B);
  \draw[->,thick] (A) -- (B) node [midway,above] {$\Psi_{K}$};
  \draw[->,thick] (C) -- (D) node [midway,above] {$\Psi_{L}$};
  \end{tikzpicture}
  \end{center}
where:
\begin{itemize}
\item \(\nu_K\) is the normal bundle of \(Y_K\subset X\); this splits
  as a direct sum of toric line bundles, due to Assumption \ref{assumption:logCalabi-Yau1-dim}.
\item \(\nu_L\) is the tangent space at \(Y_L\).
\item \(\Psi_K\) and \(\Psi_L\) are symplectic embeddings of \(\OP{nbhd}_{\nu_K}(Y_K)\) and \(\OP{nbhd}_{\nu_L}(Y_L)\) respectively into \(X\), with \(N_K = \OP{Image}(\Psi_K)\) and \(N_L = \OP{Image}(\Psi_L)\).
\item The vertical arrows are the natural inclusions.
\end{itemize}
 
By shrinking the domain of the regularisation, we can assume that
$N_K$ is invariant with respect to the torus action of \(\nu_K\), and that the inclusion $\OP{nbhd}_{\nu_L}(Y_L) \hookrightarrow \OP{nbhd}_{\nu_K}(Y_K)$ is toric. Therefore, we obtain other commutative squares
  \begin{center}
  \begin{tikzpicture}
  \node (B) at (3,2) {\(B_K \subset M(K)_{\mathbb{R}}\)};
  \node (A) at (0,2) {\(\OP{nbhd}_{\nu_K}(Y_K)\)};
  \node (D) at (3,0) {\(B_L \subset M(L)_{\mathbb{R}}\)};
  \node (C) at (0,0) {\(\OP{nbhd}_{\nu_L}(Y_L)\)};
  \draw[->,thick] (C) -- (A);
  \draw[->,thick] (D) -- (B);
  \draw[->,thick] (A) -- (B) node [midway,above] {$\mu_{K}$};
  \draw[->,thick] (C) -- (D) node [midway,above] {$\mu_{L}$};
  \end{tikzpicture}
  \end{center}
where:
\begin{itemize}
\item \(M(K)= N(K)^{\vee}\) and \(M(L)= N(L)^{\vee}\)
are the character lattices of the torus acting on $\nu_K$ and on $\nu_L$ respectively.
\item $\mu_K$ and $\mu_L$ are toric moment maps.
\end{itemize}
  
Let \(\mathbf{B}\) be the colimit of the diagram of spaces whose
vertices are the bases \(B_K\) with \(|K|\in\{n-1,n\}\) and whose
morphisms are the inclusions \(B_L\to B_K\); that is, \(\mathbf{B}\)
is the quotient of
\( \bigsqcup_{K : |K|=n-1} B_K \sqcup \bigsqcup_{L : |L|=n} B_L \) by
the equivalence relation which identifies each \(B_L\) with its image
under the inclusion map $B_L \subset B_K$. Note that \(\mathbf{B}\) is
a strong deformation retract of $\Cone(\mathcal{D}(Y)) \setminus Z$.

Let $\mathcal{N} \subset X$ be the union of $N_K$ for all $K \subset I$ of size $n-1$ with $Y_K \neq \emptyset$.
\begin{Proposition}
There exists a Lagrangian torus fibration \[\phi: \mathcal{N} \to \mathbf{B}.\]
with only toric singularities, along the boundary of \(\mathbf{B}\).
\end{Proposition}  
\begin{proof}
Set $\phi(x)=\mu_K \circ \Psi_K^{-1}(x)$ for all $x \in N_K$. The commutative squares above ensure that the maps $\mu_K \circ \Psi_K^{-1}$ agree on plumbing regions, and so that $\phi$ is well-defined. Moreover, $\phi$ is a Lagrangian fibration with only toric singularity along the boundary of \(\mathbf{B}\), as $\mu_K$ are so.
\end{proof}

By the Arnold-Liouville theorem, \(\mathbf{B}^{\circ}\) inherits an integral
affine structure, and up to fibred symplectomorphisms, the Lagrangian
torus fibration $\phi$ can be locally identified with the fibration
\[ T^{\vee}\mathbf{B^{\circ}}/\Lambda^{\vee} \to \mathbf{B}^{\circ},
\]
where $\Lambda$ is the integral lattice of tangent vectors defining
the integral affine structure on $\mathbf{B}^{\circ}$; see Definition
\ref{dfn:Zaff}(3). The integral affine structure on \(\mathbf{B}^{\circ}\) actually extends to \(\mathbf{B}\), since $\phi$ has only toric singularities.

We study now the monodromy of the integral affine structure on
\(\mathbf{B}\). The key observation is that the toric inclusions
$N_L \to N_K$ are induced by the same linear maps
$\beta_{LK}: N(L) \to N(K)$, defined in Section
\ref{sec:non-archmonodromy}, or dually by
\((\beta_{LK}^{-1})^{t}: M(L) \to M(K).\)
 
Fix a point $b_0 \in B_K$ of $\mathbf{B}$. As above, the integral
affine structure of $\mathbf{B}$ at $b_0$ can be identified with the
character lattice $M(K)$. Therefore, the monodromy of $\phi$
 \[\rho_{\text{Lagr}}:
   \pi_1(\mathbf{B}) \simeq \pi_1(\Cone(\mathcal{D}(Y))\setminus Z)
   \to \operatorname{GL}(M(K)),\] is given
 by
 \[\rho_{\text{Lagr}}([\gamma])= (\beta_{L_{n-1}K}^{-1})^t \circ
   (\beta_{L_{n-1}K_{n-2}})^t \ldots \circ (\beta_{L_1 K_2}^{-1})^t
   \circ (\beta_{L_1 K})^t \] for any loop
 $\gamma \in \pi_1(\mathbf{B}) \simeq
 \pi_1(\Cone(\mathcal{D}(Y))\setminus Z)$. In particular, we have that
 \[\rho_{\text{Lagr}}([\gamma])= (\rho_{\text{Lagr}}([\gamma])^{-1})^t.\]
 
 \begin{Proposition}\label{prop:monodromy}
   The natural pairing
   $\langle \cdot , \cdot \rangle : N(K) \times M(K) \to \mathbb{Z}$
   induces the identification
   \(\rho_{\text{non-arch}}=(\rho_{\text{Lagr}}^{-1})^{t}\).
 \end{Proposition}
  
\begin{example}
  Let $X \subset \mathbb{P}^3$ be a smooth cubic surface and
  $Y = \bigcup^3_{i=1}Y_i$ be a hyperplane section which consists of
  three lines pairwise intersecting in three double points. The normal
  bundle of $Y_i$ in $X$ is isomorphic to the line bundle
  $\mathcal{O}_{\mathbb{P}^1}(-1)$. In this case, $B_L$ are squares,
  and $B_{12}$, $B_{23}$, $B_{13}$ are right
  trapezia. 
  In the picture, we draw also the colimit \(\mathbf{B}\), up to an
  identification which consists of an affine transformation whose
  linear part is $- \mathrm{id}$. By definition, this linear map is
  the monodromy of the integral affine structure of $\mathbf{B}$ along
  a generator of its fundamental group
  $\pi_1(\mathbf{B}) \simeq \pi_1(\mathcal{D}(Y)) \simeq
  \pi_1(S^1)\simeq \mathbb{Z}$.  \vspace{0.2 cm}
  \begin{center}
  \begin{tikzpicture}[scale=0.70]
  \draw (0,1.5) --(1, 1.5);
  \draw (1, 1.5) --(1, 5.5);
  \draw (0, 4.5) --(1, 5.5);
  \draw (0, 1.5) --(0, 4.5);
  \draw (1,1) -- (5,1);
  \draw (2,0) -- (5,0);
  \draw (1,1) -- (2,0);
  \draw (5,1) -- (5,0);
  \draw (5,1.5) -- (6, 2.5);
  \draw (5, 1.5) -- (5, 5.5);
  \draw (5, 5.5) -- (6, 5.5);
  \draw (6,2.5) -- (6, 5.5);
  \draw (10, 0) --(10, 3) -- (11, 4) -- (11, 1) -- (13, 1) -- (16, 4) -- (16, 3)--(13, 0) --(10, 0);
  \draw[thick, fill=blue!15] (10,4.5)--(10, 5.5)--(11, 5.5)--(11, 4.5)--(10, 4.5);
  \draw[thick, fill=blue!15] (15,4.5)--(15, 5.5)--(16, 5.5)--(16, 4.5)--(15, 4.5);
   \draw[thick, fill=blue!15] (10,2)--(11,3)--(11,4)--(10,3)--(10,2);   
    \draw[thick, fill=blue!15] (15,3)--(15,2)--(16,3)--(16,4)--(15,3);
     \draw[fill=blue!15] (0,4.5)--(1,5.5)--(1,4.5)--(0,3.5);
        \draw[fill=blue!15] 
        (5,5.5)--(6,5.5)--(6,4.5)--(5,4.5)--(5,5.5);
        \draw[fill=red!15] (0,1.5)--(1,1.5)--(1,2.5)--(0, 2.5)--(0,1.5);
        \draw[fill=red!15] (1,1)--(2,1)--(3,0)--(2,0)--(1,1);
        \draw[fill=red!15] (10,0)--(11,0)--(11,1)--(10,1)--(10,0);
        \draw[fill=green!15] (4,0)--(5,0)--(5,1)--(4,1)--(4,0);
        \draw[fill=green!15] (5, 1.5)--(6, 2.5)--(6,3.5)--(5, 2.5)--(5, 1.5);
        \draw[fill=green!15] (12, 0)--(13,0)--(14,1)--(13,1)--(12,0);
    \draw[thick, ->] (10.5, 4.25)--(10.5, 3.75);
    \draw[thick, ->] (15.5, 4.25)--(15.5, 3.75);
    \draw[thick, ->] (14.75, 5)--(11.25, 5);
    \node (A) at (13, 5.25) {\(\rho_{\mathrm{Lagr}}=-\mathrm{id}\)};
  \draw[thick, decoration={markings, mark=at position 0.65 with {\arrow{>>}}},postaction={decorate}] (10,4.5)--(10, 5.5);
  \draw[thick, decoration={markings, mark=at position 0.65 with {\arrow{>>}}},postaction={decorate}] (10,2)--(10, 3);
  \draw[thick, decoration={markings, mark=at position 0.65 with {\arrow{>}}},postaction={decorate}] (10,4.5)--(11, 4.5);
  \draw[thick, decoration={markings, mark=at position 0.65 with {\arrow{>}}},postaction={decorate}] (10,2)--(11, 3);
    \draw[thick, decoration={markings, mark=at position 0.65 with {\arrow{>}}},postaction={decorate}] (16, 5.5) --(15, 5.5);
    \draw[thick, decoration={markings, mark=at position 0.65 with {\arrow{>>}}},postaction={decorate}] (16,3)--(15, 2);
    \draw[thick, decoration={markings, mark=at position 0.65 with {\arrow{>>}}},postaction={decorate}] (16,5.5)--(16, 4.5);
    \draw[thick, decoration={markings, mark=at position 0.65 with {\arrow{>}}},postaction={decorate}] (16,3)--(16, 4);
  \node[black] at (-1.25,2.5) [above right] {\(B_{12}\)};
  \node[black] at (7.25,3.5) [above left] {\(B_{13}\)};
  \node[black] at (3.5, -0.75) {\(B_{23}\)};
   \node[black] at (13, -0.75) {\(\mathbf{B}\)};
   \node[black] at (9.25,5) {\(B_{123}\)};
  \end{tikzpicture}
  \end{center}  
\end{example}
 
 To conclude, we summarize the results of this section in the following theorem.
 
  \begin{thm}\label{thm:nonarchLagrafibra}
  Suppose that \((X, Y)\) is a simple normal crossing pair of maximal
  intersection and log Calabi-Yau along the irreducible components of the 1-dimensional stratum of
  \(Y\). Let \(\mathfrak{X}_{\eta}\) be the {\em analytic generic fibre} of the formal completion of \(X\) along \(Y\), and \(Z\) be the union of the cells of codimension \(\geq 2\)
  in \(\Cone(\mathcal{D}(Y))\). 
  
  There exists an affinoid torus fibration 
  \[\rho_{\mathfrak{X}_{\eta}}: \mathfrak{X}_{\eta} \to \Cone(\mathcal{D}(Y))\setminus Z, \]
  a neighbourhood $\mathcal{N}$ of the 1-dimensional stratum of $Y$ and a Lagrangian torus fibration 
  \[
  \phi: \mathcal{N} \to \mathbf{B} \subset \Cone(\mathcal{D}(Y))\setminus Z,
  \]
  where $\mathbf{B}$ is a retract of
  $\Cone(\mathcal{D}(Y))\setminus Z$, such that the monodromy of the
  integral affine structure induced by $\rho_{\mathfrak{X}_{\eta}}$
  and $\phi$ are dual in the sense of Proposition
  \ref{prop:monodromy}.
  \end{thm}

\appendix
\section{Symplectic plumbing neighbourhoods}
\label{app}

\begin{lem}\label{lma:app1}
  Suppose you have two symplectic normal crossing divisors
  \(Y\subset X\) and \(Y'\subset X'\), each with pairwise
  symplectically orthogonal components, each indexed by the set
  \(I\). Suppose that there is a homeomorphism \(f\colon Y\to Y'\)
  which restricts to a symplectomorphism \(f_i\colon Y_i\to Y'_i\) on
  each component. Suppose moreover that there are symplectic bundle
  isomorphisms \(F_i\colon\nu_XY_i\to\nu_{X'}Y'_i\) with the property
  that \(F_i|_{Y_{ij}}\) agrees with
  \(df_j\colon\nu_{Y_j}Y_{ij}\to\nu_{Y'_j}Y'_{ij}\). Then there are
  symplectomorphic neighbourhoods \(\OP{nbhd}_X(Y)\) and
  \(\OP{nbhd}_{X'}(Y')\).
\end{lem}
\begin{proof}
  By {\cite[Proposition 4.2]{MTZ}} we can construct smooth
  regularisations \(\psi_J\) of \(Y\) and \(\psi'_J\) of \(Y'\) with the
  property that the normal bundle of each stratum \(Y_J\),
  \(J\subset I\), is identified with the symplectic orthogonal
  complement \(TY_J^{\omega}\) (both in \(X\) and within all
  intermediate strata). These regularisations are compatible in the
  sense that if \(K\subset J\) then
  \(\psi_J=\psi_K\circ d_{\nu}\psi_{K;J}\) where \(d_{\nu}\psi_{K;J}\)
  denotes the linearisation of \(\psi_K\) in the normal directions to
  \(Y_J\) (see {\cite[Definition 4]{MTZSurvey}}).

  In particular, this gives regularisations of each stratum \(Y_J\)
  within each \(Y_i\) with \(i\in J\). Now, stratum by stratum
  starting at the deepest, we modify \(f_i\) by an isotopy so that the following diagrams commute:

  \begin{center}
    \begin{tikzpicture}
      \node (A) at (0,0) {\(Y_i\)};
      \node (B) at (2,0) {\(Y'_i\)};
      \node (C) at (0,-2) {\(\nu_{Y_i}Y_J\)};
      \node (D) at (2,-2) {\(\nu_{Y'_i}Y_J\)};
      \draw[->,thick] (A) -- node [midway,above] {\(f_i\)} (B);
      \draw[->,thick] (C) -- node [midway,below] {\(d_{\nu}f_i\)} (D);
      \draw[->,thick] (C) -- (A);
      \draw[->,thick] (D) -- (B);
    \end{tikzpicture}
  \end{center}
  where the vertical maps are the regularisations, and \(d_{\nu}\)
  denotes the linearisation of \(f_i\) along the normal directions to
  \(Y_J\) in \(Y_i\). The idea here is that, in local coordinates, a
  diffeomorphism can be made to coincide locally with its derivative
  by an isotopy (by isotoping its graph to coincide with the graph of
  its linearisation).

  We now use our regularisations and the symplectic bundle
  isomorphisms \(F_i\) to construct a diffeomorphism from a
  neighbourhood of \(Y\) to a neighbourhood of \(Y'\) which induces an
  isomorphism of symplectic vector bundles \(TX|_{Y}\) to
  \(TX'|_{Y'}\). The idea is to use
  \(\Psi_i:=\psi'_i\circ F_i\circ \psi_i^{-1}\) on the codomain of the
  regularisation \(\psi_i\). For this to make sense, we need
  \(\Psi_i=\Psi_j\) on all intersections. For the sake of notational
  simplicity, we illustrate this only in the case when
  \(Y=Y_1\cup Y_2\).

  \begin{center}
    \begin{tikzpicture}
      \filldraw[red,opacity=0.3] (-0.5,-1) -- (-0.5,1) -- (0.5,1) -- (0.5,-1) -- cycle;
      \filldraw[blue,opacity=0.3] (-1,-0.5) -- (1,-0.5) -- (1,0.5) -- (-1,0.5) -- cycle;
      \filldraw[purple,opacity=0.3] (0.5,0.5) -- (-0.5,0.5) -- (-0.5,-0.5) -- (0.5,-0.5);
      \draw (-0.5,1) -- (-0.5,0.5) -- (-1,0.5);
      \draw (-1,-0.5) -- (-0.5,-0.5) -- (-0.5,-1);
      \draw (0.5,-1) -- (0.5,-0.5) -- (1,-0.5);
      \draw (1,0.5) -- (0.5,0.5) -- (0.5,1);
      \draw (-1,0) -- (1,0);
      \draw (0,-1) -- (0,1);
      \filldraw[red,opacity=0.3] (2,1) -- (2,3) -- (3,3) -- (3,1) -- cycle;
      \filldraw[blue,opacity=0.3] (1.5,-1) -- (3.5,-1) -- (3.5,-2) -- (1.5,-2) -- cycle;
      \filldraw[purple,opacity=0.3] (2,0.5) -- (3,0.5) -- (3,-0.5) -- (2,-0.5);
      \draw (2.5,1) -- (2.5,3);
      \draw (1.5,-1.5) -- (3.5,-1.5);
      \draw (2,0) -- (3,0);
      \draw (2.5,-0.5) -- (2.5,0.5);
      \filldraw[red,opacity=0.3] (8,-1) -- (8,1) -- (9,1) -- (9,-1) -- cycle;
      \filldraw[blue,opacity=0.3] (7.5,-0.5) -- (9.5,-0.5) -- (9.5,0.5) -- (7.5,0.5) -- cycle;
      \filldraw[purple,opacity=0.3] (9,0.5) -- (8,0.5) -- (8,-0.5) -- (9,-0.5);
      \draw (8,1) -- (8,0.5) -- (7.5,0.5);
      \draw (7.5,-0.5) -- (8,-0.5) -- (8,-1);
      \draw (9,-1) -- (9,-0.5) -- (9.5,-0.5);
      \draw (9.5,0.5) -- (9,0.5) -- (9,1);
      \draw (7.5,0) -- (9.5,0);
      \draw (8.5,-1) -- (8.5,1);
      \filldraw[red,opacity=0.3] (5.5,1) -- (5.5,3) -- (6.5,3) -- (6.5,1) -- cycle;
      \filldraw[blue,opacity=0.3] (5,-1) -- (7,-1) -- (7,-2) -- (5,-2) -- cycle;
      \filldraw[purple,opacity=0.3] (5.5,0.5) -- (6.5,0.5) -- (6.5,-0.5) -- (5.5,-0.5);
      \draw (6,1) -- (6,3);
      \draw (5,-1.5) -- (7,-1.5);
      \draw (5.5,0) -- (6.5,0);
      \draw (6,-0.5) -- (6,0.5);
      \draw[->,thick] (2.25,2) -- node [midway,above] {\(\psi_1\)} (0.2,0.8);
      \draw[->,thick] (2.25,0.25) -- node [midway,above] {\(\psi_{12}\)} (0.25,0.25);
      \draw[->,thick] (2.0,-1.3) -- node [midway,below] {\(\psi_{2}\)} (0.8,-0.25);
      \draw[->,thick] (2.75,-0.25) -- node [midway,right] {\(d_\nu\psi_{2,12}\)} (2.75,-1.3);
      \draw[->,thick] (2.75,0.25) -- node [midway,right] {\(d_\nu\psi_{1,12}\)} (2.75,1.3);
      \draw[->,thick] (6.25,2) -- node [midway,above] {\(\psi'_1\)} (8.2,0.8);
      \draw[->,thick] (6.25,0.25) -- node [midway,above] {\(\psi'_{12}\)} (8.25,0.25);
      \draw[->,thick] (6.5,-1.3) -- node [midway,below] {\(\psi'_{2}\)} (7.8,-0.25);
      \draw[->,thick] (5.75,-0.25) -- node [midway,left] {\(d_\nu\psi'_{2,12}\)} (5.75,-1.3);
      \draw[->,thick] (5.75,0.25) -- node [midway,left] {\(d_\nu\psi'_{1,12}\)} (5.75,1.3);
      \draw[->,thick] (2.75,2) -- node [midway,above] {\(F_1\)} (5.75,2);
      \draw[->,thick] (2.75,0.1) -- node [midway,below] {\(F_{12}\)} (5.75,0.1);
      \draw[->,thick] (2.75,-1.8) -- node [midway,below] {\(F_2\)} (5.75,-1.8);
    \end{tikzpicture}
  \end{center}

  Here, \(F_{12}\) is the restriction of \(F_1\) to \(\nu_{Y_2}Y_{12}\)
  plus the restriction of \(F_2\) to \(\nu_{Y_1}Y_{12}\). Because of the
  way we modified our homeomorphism \(f\), all parts of the diagram
  commute, so we see that \(\Psi_i=\Psi_j\).
  
  Now write \(\omega\) for the symplectic form on \(X\) and
  \(\omega'\) for the pullback of the symplectic form on \(X'\) along
  this diffeomorphism. Since \(\omega\) and \(\omega'\) are symplectic
  and agree along \(Y\), the 2-forms
  \(\omega_t=\omega+t(\omega'-\omega)\) give an isotopy of symplectic
  forms in a (smaller) neighbourhood of \(Y\). By Moser's isotopy
  theorem, this means that \(\omega\) and \(\omega'\) are
  symplectomorphic when restricted to this smaller neighbourhood.
\end{proof}

\begin{Remark}\label{rmk:plumbingapp}
  Another way of thinking about this is that the neighbourhood of
  \(Y\) is obtained from the normal bundles \(\nu_XY_i\) by plumbing
  along subbundles \(\nu_XY_{I}\). The different possible
  neighbourhoods can be obtained by twisting the plumbing
  identification using symplectic gauge transformations of these
  subbundles preserving the stratification by subbundles
  \(\nu_{Y_J}Y_I\). This is analogous to plumbing 2-dimensional
  surfaces along square patches, where one has the freedom to twist
  neither, either or both of the square patches, corresponding to the
  \(\ZZ/2\times\ZZ/2\) of gauge transformations of \(\RR^2\)
  preserving the \(x\)- and \(y\)-axes.

  \begin{center}
    \begin{tikzpicture}
      \filldraw[draw=none,fill=gray,opacity=0.5] (-1,0.3) -- (1,0.3) -- (1,-0.3) -- (-1,-0.3);
      \filldraw[draw=none,fill=gray,opacity=0.5] (0.3,-1) -- (0.3,1) -- (-0.3,1) -- (-0.3,-1);      
      \draw (-1,0.3) -- (1,0.3);
      \draw (-1,-0.3) -- (1,-0.3);
      \draw (-0.3,-1) -- (-0.3,1);
      \draw (0.3,-1) -- (0.3,1);
      \node at (-0.15,0.15) {A};
      \node at (0.15,0.15) {B};
      \node at (-0.15,-0.15) {C};
      \node at (0.15,-0.15) {D};
      \draw[red] (-1,0) -- (1,0);
      \draw[blue] (0,-1) -- (0,1);
      \begin{scope}[shift={(3,0)}]
        \filldraw[draw=none,fill=gray,opacity=0.5] (-1,0.3) -- (1,0.3) -- (1,-0.3) -- (-1,-0.3);
        \filldraw[draw=none,fill=gray,opacity=0.5] (-0.3,-1) -- (0.3,-0.3) -- (0.3,0.3) -- (-0.3,1) -- (0.3,1) -- (-0.3,0.3) -- (-0.3,-0.3) -- (0.3,-1);
        \draw (-1,0.3) -- (1,0.3);
        \draw (-1,-0.3) -- (1,-0.3);
        \draw (-0.3,-1) -- (0.3,-0.3) -- (0.3,0.3) -- (-0.3,1);
        \draw (0.3,-1) -- (0.04,-0.68);
        \draw (-0.04,-0.62) -- (-0.3,-0.3) -- (-0.3,0.3) -- (-0.04,0.62);
        \draw (0.04,0.68) -- (0.3,1);
        \node at (0.15,0.15) {\reflectbox{A}};
        \node at (-0.15,0.15) {\reflectbox{B}};
        \node at (0.15,-0.15) {\reflectbox{C}};
        \node at (-0.15,-0.15) {\reflectbox{D}};
        \draw[red] (-1,0) -- (1,0);
        \draw[blue] (0,-1) -- (0,1);
      \end{scope}
      \begin{scope}[shift={(0,-3)}]
        \filldraw[draw=none,fill=gray,opacity=0.5] (0.3,-1) -- (0.3,1) -- (-0.3,1) -- (-0.3,-1);
        \filldraw[draw=none,fill=gray,opacity=0.5] (-1,-0.3) -- (-0.3,0.3) -- (0.3,0.3) -- (1,-0.3) -- (1,0.3) -- (0.3,-0.3) -- (-0.3,-0.3) -- (-1,0.3);
        \draw (0.3,-1) -- (0.3,1);
        \draw (-0.3,-1) -- (-0.3,1);
        \draw (-1,-0.3) -- (-0.3,0.3) -- (0.3,0.3) -- (1,-0.3);
        \draw (-1,0.3) -- (-0.68,0.04);
        \draw (-0.62,-0.04) -- (-0.3,-0.3) -- (0.3,-0.3) -- (0.62,-0.04);
        \draw (0.68,0.04) -- (1,0.3);
        \node at (-0.15,-0.15) {\scalebox{1}[-1]{A}};
        \node at (0.15,-0.15) {\scalebox{1}[-1]{B}};
        \node at (-0.15,0.15) {\scalebox{1}[-1]{C}};
        \node at (0.15,0.15) {\scalebox{1}[-1]{D}};
        \draw[red] (-1,0) -- (1,0);
        \draw[blue] (0,-1) -- (0,1);
      \end{scope}
      \begin{scope}[shift={(3,-3)}]
        \filldraw[draw=none,fill=gray,opacity=0.5] (-0.3,-1) -- (0.3,-0.3) -- (0.3,0.3) -- (-0.3,1) -- (0.3,1) -- (-0.3,0.3) -- (-0.3,-0.3) -- (0.3,-1);
        \filldraw[draw=none,fill=gray,opacity=0.5] (-1,-0.3) -- (-0.3,0.3) -- (0.3,0.3) -- (1,-0.3) -- (1,0.3) -- (0.3,-0.3) -- (-0.3,-0.3) -- (-1,0.3);
        \draw (-1,-0.3) -- (-0.3,0.3) -- (0.3,0.3) -- (1,-0.3);
        \draw (-1,0.3) -- (-0.68,0.04);
        \draw (-0.62,-0.04) -- (-0.3,-0.3) -- (0.3,-0.3) -- (0.62,-0.04);
        \draw (0.68,0.04) -- (1,0.3);
        \draw (-0.3,-1) -- (0.3,-0.3) -- (0.3,0.3) -- (-0.3,1);
        \draw (0.3,-1) -- (0.04,-0.68);
        \draw (-0.04,-0.62) -- (-0.3,-0.3) -- (-0.3,0.3) -- (-0.04,0.62);
        \draw (0.04,0.68) -- (0.3,1);
        \node at (-0.15,-0.15) {\scalebox{-1}[-1]{B}};
        \node at (0.15,-0.15) {\scalebox{-1}[-1]{A}};
        \node at (-0.15,0.15) {\scalebox{-1}[-1]{D}};
        \node at (0.15,0.15) {\scalebox{-1}[-1]{C}};
        \draw[red] (-1,0) -- (1,0);
        \draw[blue] (0,-1) -- (0,1);
      \end{scope}
    \end{tikzpicture}
  \end{center}
\end{Remark}

\bibliographystyle{plain}
\bibliography{construction}
\end{document}